\documentclass[reqno,11pt]{amsart}
\usepackage{amssymb, graphicx}
\usepackage[title,titletoc,header]{appendix}
\usepackage{xcolor}
\usepackage[pagewise]{lineno}

\newtheorem{lemma}{Lemma}[section]
\newtheorem{theorem}{Theorem}[section]
\newtheorem{proposition}{Proposition}[section]

\theoremstyle{definition}
\newtheorem{definition}{Definition}[section]
\theoremstyle{remark}
\newtheorem{remark}{Remark}[section]

\numberwithin{equation}{section}

\newcommand{\p}{\partial}
\newcommand{\norm}[1]{\left\Vert#1\right\Vert}
\newcommand{\dd}{\mathrm{d}}
\newcommand{\di}{\mathrm{div}}
\newcommand{\grad}{\mathrm{grad}}
\newcommand \R{\mathbb{R}}
\newcommand \T{\mathbb{T}}
\newcommand \Z{\mathbb{Z}}

\makeatletter

\newcommand{\Rmnum}[1]{\mathrm{\expandafter\@slowromancap\romannumeral#1@}}
\makeatother

\begin{document}
\title[Stabilization effect of frictions for transonic shocks]
{Stabilization effect of frictions for transonic shocks\\ in steady compressible Euler flows passing\\ three-dimensional ducts}

\author{Hairong Yuan}
\address{Hairong Yuan:
School of Mathematical Sciences, and Shanghai Key Laboratory of Pure Mathematics and Mathematical Practice,
East China Normal University, Shanghai 200241, China}
\email{hryuan@math.ecnu.edu.cn}

\author{Qin Zhao}
\address{Qin Zhao (Corresponding author): School of Mathematical Sciences, Shanghai Jiao Tong University,
	Shanghai 200240, China}
\email{zhao@sjtu.edu.cn}



\begin{abstract}
Transonic shocks play a pivotal role in designation of supersonic inlets and ramjets. For the three-dimensional steady non-isentropic compressible Euler system with frictions,  we had constructed a family of transonic shock solutions in rectilinear ducts with square cross-sections, and this paper is devoted to proving rigorously that a large class of these transonic shock solutions  are stable, under multidimensional small perturbations of the upcoming supersonic flows and back pressures at the exits of ducts in suitable function spaces. This manifests that friction has a stabilization effect on transonic shocks in ducts, in consideration of previous works have shown that transonic shocks in purely steady Euler flows are not stable in such ducts. Except its implications to applications, since frictions lead to a stronger coupling between the elliptic and hyperbolic parts of the three-dimensional steady subsonic Euler system, we develop the framework established in previous works to study more complex and interesting Venttsel problems of nonlocal elliptic equations.
\end{abstract}
\maketitle
\tableofcontents

\section{Introduction}\label{sec1}
This paper is a continuation of previous works \cite{Yuan2006,Yuan2007,ChenYuan2008,Yuan2008-2,LiuYuan2008,ChenYuan2013,
FangLiuYuan2013,LiuXuYuan2016} on a systematic investigation of existence, stability and uniqueness of transonic shocks (i.e. normal shocks) in steady compressible Euler flows. Starting from the work \cite{chen-feldman-2003} of Chen and Feldman, the study of steady transonic shocks has attracted many authors (see also  \cite{BaeFeldman2011,LiXinYin2013,XinYin2005} and references therein), not only due to its important applications to supersonic inlet and ramjet (see, for example, \cite[Chapters 3,11]{Ra2010}), but also the difficulties it involved in mathematics, such as free-boundary, nonlinear equations of elliptic-hyperbolic composite-mixed type. Interestingly, although transonic shocks could be observed in experiments in a seemingly rectilinear duct (see for example, photograph 225 in \cite{vandyke}), the previous mathematical studies have shown that transonic shock solutions to steady compressible Euler system in straight ducts are not stable under perturbations of the back pressures posed at the exits, or the up-stream supersonic flows (see, for instance, \cite{Yuan2006, ChenYuan2008,FangLiuYuan2013}).

To solve this paradox, many authors considered geometric effects \cite{Yuan2008-2}. For a ``non-isentropic potential flow model" proposed by Bae and Feldman, the stability of transonic shocks in divergent nozzles were  proved in \cite{BaeFeldman2011}. For the two-dimensional outward cylindrical full compressible Euler flows, the stability of transonic shocks was shown in \cite{LiuYuan2008}, where the authors discovered many interesting nonlocal elliptic problems coming from interactions of the elliptic part and hyperbolic part of the steady subsonic Euler system. The case of Euler flows in two-dimensional divergent nozzles  was proved in \cite{LiXinYin2013}. For the more difficult three-dimensional steady  compressible Euler  flows, in \cite{LiuXuYuan2016} the authors proved stability of outward spherical transonic shocks. These works demonstrate that geometry (expanding of area of the flow tube) has a stabilization effect on transonic shocks. It is natural to ask whether friction, which is considered as an important factor in engineering for nozzle flows,    has a stabilization effect on transonic shocks. If it is, then one partially solves the paradox mentioned above. This is exactly the purpose of this work.

The main difficulties of studying stationary transonic shocks  come from the facts that the full compressible Euler equations of steady subsonic flows consist a nonlinear system of conservation laws of elliptic-hyperbolic composite-mixed type, and the shock-front is a free-boundary. In \cite{LiuXuYuan2016}, the authors have established a framework to decompose the steady Euler system, as well as the Rankine-Hugoniot jump conditions,  that works in a general product Riemannian manifold. The Euler system is rewritten as four transport equations plus a second-order elliptic equation of pressure, for them the lower-order terms are coupled. In this paper,  although general tools from differential geometry are not necessary, frictions do lead to more stronger coupling of the elliptic parts and hyperbolic parts. So we need to treat more general Venttsel problems of nonlocal elliptic equations.

In a previous paper \cite{YuanZhao}, we have considered steady subsonic compressible Euler flows in a duct with frictions. In that work we have also shown existence of a family of special transonic shock solutions by considering fluid flows only in the axial direction of the ducts, and the frictions acting on the axial direction opposite to that of the flow. So for simplicity, we sometimes just borrow computations from \cite{YuanZhao}. However, to make the paper more readable, there are some necessary repetitions. As we mentioned in \cite{YuanZhao}, a single friction term changes drastically the behavior of solutions of Euler system, so although many expressions in this paper look quite similar to those appeared in \cite{LiuXuYuan2016}, there are some major differences in details, and we had carried out careful computations.

We remark that there are also many works on transonic shocks and frictional flows in the time-dependent case, but mainly on the so-called quasi-one-dimensional model, which is a hyperbolic system of balance laws, see for example, \cite{Liu1982-1,Liu1982-2,RauchXieXin2013,Tsuge2015,CHHQ2016} and references therein. If the friction force depend linearly on momentum, one may also consult \cite{HMP2005,HuangPanWang2011} and references therein for large time behavior of weak solutions.

In the rest of this section we formulate the transonic shock problem (T), and state the main result of this paper, namely Theorem \ref{thm201}. In Section \ref{sec3} we reformulate problem (T) by using a decomposition of the system \eqref{eq101}-\eqref{eq103} established in \cite{LiuXuYuan2016}. In Section \ref{sec4}, we study a crucial Venttsel boundary-value problem for a second-order nonlocal elliptic equation. In Section \ref{sec5}, by showing contraction of a nonlinear mapping, we prove Theorem \ref{thm201}.

\subsection{The transonic shock problem}\label{sec11}
We now formulate the transonic shock problem in a more technical way.
As in all of the previous works, we consider polytropic gases:
$$p=A(s)\rho^\gamma,$$
where $p$ is the scalar pressure, $\rho$ is the density of mass, $\gamma>1$ is the adiabatic exponent, and $s$ is the entropy per unit mass, with $A(s)=k_0\exp(s/c_\nu)$, and $k_0, c_v$ are positive constants.  The sonic speed is given by $c=\sqrt{\gamma p/\rho}$.

In the Descartesian coordinates $(x^0, x^1,x^2)$ of the Euclidean space $\R^3$, let $D=\{(x^0,x^1,x^2): x^0\in(0,L), (x^1,x^2)\in(0,\pi)\times(0,\pi)\}$ be a rectilinear duct with length $L$ and constant square cross-section, where the gas flows along the positive $x^0$-direction. Hence we call $D_0=\{(x^0,x^1,x^2): x^0=0, (x^1,x^2)\in(0,\pi)\times(0,\pi)\}$ and $D_L=\{(x^0,x^1,x^2): x^0=L, (x^1,x^2)\in(0,\pi)\times(0,\pi)\}$
respectively the entry and exit of the duct. To avoid technical difficulties arose by the lateral walls, as in \cite{chen-feldman-2003,ChenYuan2008,ChenXie2014,Weng2015}, by assuming the upstream supersonic flows and the back pressures have some symmetric properties with respect to the walls $[0,L]\times\p[0,\pi]^2$, we may suppose the flows are periodic in $x^1, x^2$-directions with periods $2\pi$. The details are presented in \cite[p.528, p.552]{ChenYuan2008}, so we omit them here.

Let $u=(u^0, u^1, u^2)^\top$ be the velocity of the gas flows. Recall that the flow is called {\it subsonic} at a point if the Mach number $M=|u|/c$ is less than $1$ there, and {\it supersonic} if $M>1$. Then as explained in \cite{YuanZhao}, the motion of compressible Euler flows with frictions is governed by the following equations ({\it cf.} \cite{CF1976, Dafermos2010, Shapiro1953}):
\begin{align}
\di (\rho u\otimes u)+\grad\, p-\rho \mathfrak{b}=&0,\label{eq101}\\
\di (\rho u)=&0,\label{eq102}\\
\di (\rho E u)-\rho \mathfrak{b}\cdot u=&0,\label{eq103}
\end{align}
where `$\di$' and `$\grad$' are respectively the standard divergence and gradient operator in $\mathbb{R}^3$; $E\triangleq\frac{1}{2}|u|^2+\frac{\gamma}{\gamma-1}\frac{p}{\rho}$  is the so-called  {\it Bernoulli constant},  and $\mathfrak{b}=(-\mu (u^0)^2,0,0)^\top$ represents the force of friction acting on per unit mass of gas with a positive constant $\mu$.  These equations are the conservation of momentum, mass and energy, respectively.

Let $\Omega=\{(x^0, x^1,x^2): x^0\in(0, L),x'=(x^1,x^2)\in\T^2\}$ be the duct we consider henceforth, where $\T^2=\mathbb{R}^2/(2\pi\mathbb{Z}^2)$ is the flat $2$-torus, with a coordinates $x'=(x^1,x^2), x^1,x^2\in[0,2\pi)$. Then $\p\Omega$,  the boundary of $\Omega$, is given by $\Sigma_0\cup \Sigma_1$, with $\Sigma_0=\{0\}\times \T^2$ and $\Sigma_1=\{L\}\times \T^2$. For $u=(u^0, u^1,u^2)^\top$, we call $u^0$ the {\it normal velocity} and $u'=(u^1,u^2)^\top$ the {\it tangential velocity}. We  use $U=({p}, {s},{E}, {u}')$ to represent the state of the gas in $\Omega$. Suppose that
\begin{eqnarray}\label{eq201}
S^\psi=\{(x^0, x')\in\Omega\,:\, x^0=\psi(x'), \ x'\in \T^2\}
\end{eqnarray}
is a surface, where $\psi: \T^2\to\Omega$ is a $C^1$ function. The normal vector field on $S^\psi$ is given by
\begin{eqnarray*}
n=(1, -\p_1\psi,-\p_2\psi).
\end{eqnarray*}
We also set $\Omega_\psi^-=\{x\in\Omega: x^0<\psi(x'), \ x'\in \T^2\}$ to be the {\it supersonic region}, and
$\Omega_\psi^+=\{x\in\Omega: x^0>\psi(x'), \ x'\in \T^2\}$ to be the {\it subsonic region}.

\begin{definition}[Transonic shock]
Let $\psi\in C^1(\T^2)$ and $U^\pm\in C^1(\Omega_\psi^\pm)\cap C(\overline{\Omega_\psi^\pm})$. We say that $U=(U^-, U^+; \psi)$ is a {\it transonic shock solution}, if
\begin{itemize}
\item[1)] $U^\pm$ solve the system \eqref{eq101}-\eqref{eq103} in $\Omega_\psi^\pm$ in the classical sense;
\item[2)] $U^-$ is supersonic, and $U^+$ is subsonic;
\item[3)] The following Rankine--Hugoniot jump conditions (R--H conditions) hold across $S^\psi$:
\begin{align}
\left.[\rho(u\cdot n)u +pn]\right.=0,& \label{eq203}\\
\left.[\rho (u\cdot n)]\right.=0,& \label{eq204}\\
\left.[\rho (u\cdot n)E]\right.=0,& \label{eq205}
\end{align}
where $u\cdot n$ is the standard inner product of the vectors $u,n\in\mathbb{R}^3$, and  $[f(U)]\triangleq f(U^+|_{S^\psi})-f(U^-|_{S^\psi})$ denotes the jump of a quantity $f(U)$ across $S^\psi$;
\item[4)] There holds the following physical entropy condition
\begin{eqnarray}\label{eq206}
[p]=p^+|_{S^\psi}-p^-|_{S^\psi}>0.
\end{eqnarray}
\end{itemize}
\end{definition}

By the definition we infer that a transonic shock solution is a weak entropy solution of the steady Euler system \eqref{eq101}-\eqref{eq103} ({\it cf.} Section 4.3 and Section 4.5 in \cite{Dafermos2010}).

To formulate the transonic shock problem, we need to specify boundary conditions, which are similar to the previous works.

Since the flow $U^-$ is supersonic near the entry $\Sigma_0$, we shall propose the following Cauchy data:
\begin{eqnarray}\label{eq207}
U=U_0^-(x') \qquad \text{on}\quad \Sigma_0.
\end{eqnarray}
Here we require that $(u^0)_0^->c^-_0$ to make sure the steady Euler system \eqref{eq101}-\eqref{eq103} is symmetric hyperbolic in the positive $x^0$-direction on $\Sigma_0$.

On the exit $\Sigma_1$, from considerations in engineering, as in the studies of subsonic flows, we require that
\begin{eqnarray}\label{eq208}
p=p_1(x')\qquad \text{on}\quad \Sigma_1,
\end{eqnarray}
where $p_1$ is a given function defined on $\mathbb{T}^2$.

\medskip
\fbox{
\parbox{0.90\textwidth}{
Problem (T): Find a transonic shock solution in $\Omega$ which satisfies the  boundary conditions \eqref{eq207} and \eqref{eq208} pointwisely.}}

\subsection{Main result}\label{sec12}
The existence of a family of special transonic shock solutions $U_b=(U_b^-, U_b^+; r_b)$ with the position $r_b$ of the shock (which are called as {\it background solutions} in the sequel) to Problem (T) has been established in \cite{YuanZhao}. Recall that the background solutions $U_b$,  which depend  only on $x^0$,  satisfy the following ordinary differential equations (see (2.2)--(2.4) in \cite{YuanZhao}):
\begin{eqnarray}\label{eqback}
\frac{\dd u^0_b}{\dd x^0}=\frac{\mu M_b^2}{1-M_b^2}u^0_b,\qquad
\frac{\dd p_b}{\dd x^0}=\frac{\mu\gamma M_b^2}{M_b^2-1}p_b,\qquad
\frac{\dd \rho_b}{\dd x^0}=\frac{\mu M_b^2}{M_b^2-1}\rho_b.
\end{eqnarray}

\begin{remark}\label{rem11new}
For given $L$ less than a maximal length, for which the flow at the exit is still subsonic, we note that  $U_b^+$  actually depends analytically on the parameters $\gamma>1, \mu\ge 0, r_b\in (0, L]$ and  $U_b^-(0)$.   In addition, it is important to note that the subsonic Fanno flow $U_b^+$ could be extended analytically to $[r_b-h_b, L]\times \T^2$ via these equations, for a small positive constant $h_b$ depending solely on these parameters.

We may also imagine that the exit $\{x^0=L\}$ lies on the  left-hand side of the shock-front,  namely,  $r_b>L$ but close to $L$:
The flow is supersonic on $0\le x^0\le r_b$, flows from left to right, and jump to subsonic at $r_b$, then flows to the left (for $x^0<r_b$), along the subsonic solution of the differential equations \eqref{eqback} with initial data $U_b^+(r_b)$ at $x^0=r_b$. Although such a flow pattern is obviously not possible in reality, it is mathematically justified, just like studying multi-valued analytic functions on Riemannian surfaces.
This fact will be used in the proof of Lemma \ref{lem401} later.
\end{remark}

In this paper we mainly concern existence of general transonic shock solutions that are obtained by multidimensional perturbations of background solutions. The main theorem is as follows.

\begin{theorem}\label{thm201}
Suppose that a background solution $U_b$  satisfies the S-Condition (see Remark \ref{rm11} below), and  $\alpha\in(0,1)$. There exist $\varepsilon_0$ and $C_*$ depending only on $U_b$ and $\gamma, \alpha, L$ such that, if the upcoming supersonic flow $U_0^-$ on $\Sigma_0$ and the back pressure $p_1$ on $\Sigma_1$ satisfy
\begin{eqnarray}
\norm{U_0^--U_b^-}_{C^{4}(\Sigma_0)}\le\varepsilon\le\varepsilon_0,&\label{eq209}\\
\norm{p_1-p_b^+}_{C^{3,\alpha}(\Sigma_1)}\le\varepsilon\le\varepsilon_0,&\label{eq210}
\end{eqnarray}
then there exists a transonic shock solution $U=(U^-,U^+; \psi)$ to Problem (T), so that  $\psi\in
C^{4,\alpha}({\T}^2)$, $U^-\in C^{4}(\overline{\Omega_\psi^-})$, $p^+\in C^{3,\alpha}(\overline{\Omega_\psi^+})$, $u^+, \rho^+, s^+\in C^{2,\alpha}(\overline{\Omega_\psi^+})$, $u^+|_{S^\psi}, \rho^+|_{S^\psi}, s^+|_{S^\psi}\in C^{3,\alpha}(\T^2)$, and
\begin{eqnarray}\label{eq211}
\norm{\psi-r_b}_{C^{4,\alpha}({\T}^2)}\le C_*\varepsilon,&\\
\label{eq212}
\norm{U^--U^-_b}_{C^4(\overline{\Omega_\psi^-})}\le C_*\varepsilon,&\\
\label{eq213}
\norm{\left.U^+\right|_{S^\psi}-\left.U_b^+\right|_{S^\psi}}_{C^{3,\alpha}(\T^2)}+\norm{U^+-U^+_b}_{3}\le C_*\varepsilon.&
\end{eqnarray}

Furthermore, such solution is unique in the class of functions $\psi, U^-, U^+$ with
\begin{eqnarray}\label{eq214}
\norm{\psi-r_b}_{C^{3,\alpha}({\T}^2)}\le C_*\varepsilon,&\\
\label{eq215}
\norm{U^--U^-_b}_{C^4(\overline{\Omega_\psi^-})}\le C_*\varepsilon,&\\
\label{eq216}
\norm{\left.U^+\right|_{S^\psi}-\left.U_b^+\right|_{S^\psi}}_{C^{2,\alpha}(\T^2)}
+\norm{U^+-U^+_b}_{2}\le C_*\varepsilon.&
\end{eqnarray}
Here, the norm $\norm{\cdot}_k$ $(k=2,3)$ is defined by
\begin{eqnarray}\label{norm}
\norm{U}_k\triangleq \norm{p}_{C^{k,\alpha}(\overline{\Omega_\psi^+})}+
\norm{s}_{C^{k-1,\alpha}(\overline{\Omega_\psi^+})}+
\norm{E}_{C^{k-1,\alpha}(\overline{\Omega_\psi^+})}+
\sum_{\beta=1}^2\norm{u^\beta}_{C^{k-1,\alpha}(\overline{\Omega_\psi^+})}.
\end{eqnarray}
\end{theorem}

\begin{remark}\label{rm11}
The technical S-Condition is given by Definition \ref{def401} in Section \ref{sec4}. It is shown there that a large class of  background solutions $U_b$  satisfy the S-Condition.
\end{remark}

\begin{remark}
The existence and uniqueness of supersonic flow $U^-$ in $\Omega=(0,L)\times \T^2$ subjected to the initial data
$U_0^-$ satisfying \eqref{eq209} follow from the theory of semi-global classical solutions of the Cauchy problem of
quasi-linear symmetric hyperbolic systems if $\varepsilon_0$ is sufficiently small (depending on $L$, {\it cf.} \cite{BS2007}). Furthermore, there exist $C_0>0$ and $\varepsilon_0>0$  depending solely on $U_b^-(0)$ and $L$, such that
\begin{eqnarray}\label{eq217}
\norm{U^--U^-_b}_{C^{4}(\overline{\Omega})}\le C_0\varepsilon,
\end{eqnarray}
which implies \eqref{eq212}. So Problem (T) is indeed a one-phase free-boundary problem, for which the free-boundary ({\it i.e.} the shock-front) $S^\psi$ and the subsonic flow $U^+$ are to be solved. {\it For simplicity, from now on we write $U^+$ as $U$.}
\end{remark}

\section{Reformulation of Problem (T)}\label{sec3}
The following is an important theorem established in  \cite[p.703]{LiuXuYuan2016}), which is rewritten for use to our case, namely, $\Omega_\psi^+$ is now a flat manifold. So for a vector field $u$, we have $D_u =u\cdot\mathrm{grad}$. As a convention, repeated Roman indices will be summed up for $0,1,2$, while repeated Greek indices are to be summed over for $1,2$, except otherwise stated.
\begin{proposition}\label{prop31}
Suppose that $p \in C^2(\Omega_\psi^+)\cap C^1(\overline{\Omega_\psi^+})$, $\rho, u\in C^1(\overline{\Omega_\psi^+})$, and $\rho>0, u^0\ne0$ in $\overline{\Omega_\psi^+}$. Then  $p, \rho, u$ solve the system \eqref{eq101}-\eqref{eq103} in $\Omega_\psi^+$ if and only if they satisfy the following equations in $\Omega_\psi^+$:
\begin{align}
&D_uE-\mathfrak{b}\cdot u=0,\label{eq301}\\
&D_uA(s)=0,\label{eq302}\\
&D_u\left(\frac{1}{\gamma p}D_u p\right)-\di \left(\frac{1}{\rho}\grad\,
p\right)- \p_ju^k\p_ku^j+\di\,\mathfrak{b}\nonumber \\
&\qquad+L^0(\frac{1}{\gamma p}D_up+\di\,u)+L^1(D_uE-\mathfrak{b}\cdot u)\nonumber\\
&\qquad\qquad+L^2(D_uA(s))
+L^3(D_uu+\frac{1}{\rho}\grad\,p-\mathfrak{b})=0,\label{eq303}\\
&D_u u^\beta+\frac{1}{\rho}\p_\beta\, p=0,\quad\beta=1,2;\label{eq304}
\end{align}
and the boundary condition on $S^\psi$:
\begin{eqnarray}
\frac{1}{\gamma p}D_u p+\di\,u+L_1(D_uE-\mathfrak{b}\cdot u)+L_2(D_uA(s))
+L_3(D_uu+\frac{1}{\rho}\grad\,p-\mathfrak{b})=0.\label{eq305}
\end{eqnarray}
Here $L^0(\cdot)$ is a linear function, and  $L^k(\cdot)$, $L_k(\cdot)$ are smooth functions so that $L^k(0)=0, L_k(0)=0$ for $k=1,2,3$.
\end{proposition}

To formulate a tractable nonlinear Problem (T1) which is equivalent to Problem (T), we need to compute the exact expressions of \eqref{eq303} and \eqref{eq305} in Proposition \ref{prop31}. Some of the details of the computations could be found in \cite{YuanZhao}.

\subsection{The equation of pressure}\label{sec21}
We report that \eqref{eq303} is equivalent to the following second-order equation of pressure:
\begin{align}\label{eq306}
N(U)&\triangleq\Big(2E-\frac{\gamma+1}{\gamma-1}c^2\Big)\p_0^2p-c^2(\p_1^2p+\p_2^2p)
-2\mu(E-\frac{ c^2}{\gamma-1})\p_0p\nonumber\\
&\quad\quad-\frac{2}{ p}\left(E-\frac{c^2}{\gamma-1}+\frac{c^4}{4\gamma}\frac{1}{E-\frac{c^2}{\gamma-1}}\right)(\p_0p)^2
+2\mu^2\gamma p\left(E-\frac{c^2}{\gamma-1}\right)\nonumber\\
&=F_3\triangleq-\gamma p(F_1+F_2),
\end{align}
where
\begin{eqnarray}\label{eq307}
F_1=\sum_{(k,j)\ne(0,0)}\left(\frac{1}{\gamma p}u^k(u^j\p_{jk}p
+\p_ku^j\p_j p-\frac{1}{p}u^j\p_jp\p_kp)-\p_ju^k\p_ku^j
+\frac{1}{\rho^2}\delta_{kj}\p_k\rho\p_j p\right);
\end{eqnarray}
$\delta_{kj}$ is the standard Kronecker delta, and
\begin{align}\label{eq308}
F_2&=-\left((u^1)^2+(u^2)^2\right)
	\left\{\frac{1}{\gamma p}\p_0^2p
+\frac{(\p_0p)^2}{\gamma p^2}\left(-1+\frac{c^4}{\gamma}
\frac{1}{2E-\frac{2c^2}{\gamma-1}}\frac{1}{2E-	 \left((u^1)^2+(u^2)^2\right)-\frac{2c^2}{\gamma-1}}\right)\right\}\nonumber\\
&\qquad+\frac{1}{\gamma p}\left\{\left(\mu\big((u^1)^2+(u^2)^2\big) -u^\beta\p_{\beta}u^0\right)\p_0p+\frac{1}{u^0}\rho^{\gamma-1}u^\beta\p_\beta A(s)\p_0p\right\}\nonumber\\
&\qquad\qquad-\frac{1}{(u^0)^2}\left(u^\beta\p_\beta u^0+\frac{2}{\rho}\p_0p\right)u^\beta\p_\beta u^0-\mu^2\Big((u^1)^2+(u^2)^2\Big).
\end{align}

\subsection{The boundary conditions}\label{sec22}
The expression \eqref{eq305} is a nonlinear condition for pressure:
\begin{eqnarray}\label{eq309}
\p_0p -\mu\gamma p\frac{ (u^0)^2}{(u^0)^2-c^2}=G_1+G_2,
\end{eqnarray}
with
\begin{align}\label{eq310}
G_1\triangleq&-\frac{1}{\left(\frac{u^0}{c}\right)^2-1}\rho u^0\p_\beta u^\beta,\\
G_2\triangleq&-\frac{1}{\left(\frac{u^0}{c}\right)^2-1}\left\{{u^0}\left(\frac{1}{c^2}+\frac{1}{(u^0)^2}
\right)u^\beta\p_\beta p\right.\nonumber\\
&\qquad\left.+\frac{\rho}{u^0}u^\beta \left(\frac{1}{\gamma-1}\rho^{\gamma-1}\p_\beta A(s)+u^\sigma \p_\sigma u^\beta-\p_\beta E\right)\right\}.\label{eq311}
\end{align}

\subsection{Decomposition of R-H conditions}\label{sec23}
By the definition of shock-front, the mass flux $m\triangleq\rho (u\cdot n)|_{S^\psi}=\rho^- (u^-\cdot n)|_{S^\psi}\ne0$ (otherwise it is called as a contact discontinuity). So from \eqref{eq204} and \eqref{eq205}, we infer that $E|_{S^\psi}=E^-|_{S^\psi}$, while
\eqref{eq204} and \eqref{eq205} may be written as
\begin{eqnarray}
[m]=0,\quad [E]=0.\label{eq312}
\end{eqnarray}
The conservation of momentum  \eqref{eq203} shall be decomposed as
\begin{align}
\left.[mu^0+p]\right.&=0,\label{eq313}\\
\left.[m{u}^\beta-p\,\p_\beta\psi]\right.&=0,\qquad \beta=1,2.\label{eq314}
\end{align}

If $[p]>0$ (which is guaranteed by the physical entropy condition satisfied by the background solution, and the small perturbation estimate \eqref{eq213} to be established),  from \eqref{eq314} we  solve that
\begin{eqnarray}\label{eq315}
\p_\beta\psi=\left.\frac{m[{u}^\beta]}{[p]}\right|_{S^\psi}
=\mu_0({u}^\beta|_{S^\psi})+g_0^\beta,\qquad \beta=1,2,
\end{eqnarray}
with
\begin{align*}
\mu_0&=\left.\frac{(\rho u^0)_b}{p_b^+-p_b^-}\right|_{x^0=r_b}>0,\\
g_0^\beta&=g_0^\beta(U,U^-, D\psi)\triangleq\left.\frac{m[{u}^\beta]}{[p]}\right|_{S^\psi}-
\mu_0({u}^\beta|_{S^\psi}).
\end{align*}
We note that $g_0^\beta$ is a higher-order term (see Definition \ref{def302} below), which depends on $U|_{S^\psi}, U^-|_{S^\psi}$ and $D\psi$.

Thus, the R-H conditions \eqref{eq203}-\eqref{eq205} are equivalent to \eqref{eq312}, \eqref{eq313} and \eqref{eq315}, if $[p]\ne0$.

\begin{remark}\label{rem301}
By \eqref{eq315}, it is necessary that $\p_2\p_1\psi-\p_1\p_2\psi=0$, which implies
\begin{align}\label{eq318}
\p_2({u}^1|_{S^\psi})-\p_1({u}^2|_{S^\psi})=&-\frac{1}{\mu_0}(\p_2g_0^1-\p_1g_0^2).
\end{align}
Since $\psi$ is well-defined on $\T^2$ , there shall hold
\begin{eqnarray}
\int_{0}^{2\pi}\Big(\mu_0({u}^1|_{S^\psi})+g_0^1\Big)(\psi(s,\pi),s,\pi)\,\dd s=0,\label{eq319}\\
\int_{0}^{2\pi}\Big(\mu_0({u}^2|_{S^\psi})+g_0^2\Big)(\psi(\pi,s),\pi,s)\,\dd s=0.\label{eq320}
\end{eqnarray}
On the contrary, since the first Betti number of $\T^2$ is $2$, by de Rham's Theorem and Hodge Theorem, \eqref{eq318}-\eqref{eq320} are also sufficient for the existence of a unique function $\psi^p$ on $\T^2$ so that \eqref{eq315} holds, and $\int_{\T^2} \psi^p\dd x^1\dd x^2=0$. The function $\psi^p$ is called the {\it profile} of the surface $S^\psi$ defined by \eqref{eq201}, and the constant $r^p\triangleq\psi-\psi^p$ is called the {\it position} of $S^\psi$. As known from previous work, and will be illustrated in this paper, $\psi^p$ is determined by R-H conditions, while $r^p$ is determined by an integral-type solvability condition derived from the Euler equations. We note that conditions like \eqref{eq319} and \eqref{eq320} do not appear in the cases of spherical symmetric flows considered in \cite[p.730]{LiuXuYuan2016}, which exhibit the significant influences of topology in the studies of transonic shock problems.
\end{remark}

\subsection{Problem (T1)}
In this subsection, we separate the linear parts from the nonlinear equations \eqref{eq301}, \eqref{eq306}, \eqref{eq309},  the R--H conditions \eqref{eq312}, \eqref{eq313}, and write them in the form $$\mathcal{L}(U-U_b^+,\psi-r_b)=\mathcal{N}(U^--U_b^-, U-U_b^+, \psi),$$ where $\mathcal{L}$ is a linear operator, and $\mathcal{N}(U^--U_b^-, U-U_b^+, \psi)$ consist of certain higher-order terms defined below.
\begin{definition}\label{def302}
Let $\hat{U}=U-U_b^+.$ A higher-order term is an expression
that contains either

(i) $U^--U_b^-$ and its first-order derivatives;

\noindent
or

(ii) the products of $\psi^p,\ r^p-r_b, \hat{U}$, and their
derivatives $D\hat{U},D^2\hat{U}, D\psi,D^2\psi$, and $D^3\psi$,
where $D^k$ is a $k^{th}$-order derivative operator.
\end{definition}

\subsubsection{Linearization of Bernoulli law}\label{sec241}
After straightforward calculations,  \eqref{eq301} is equivalent to
\begin{align}\label{eq322}
D_u \hat{E}+2\mu u^0 \hat{E}=&\frac{2\mu u^0 }{\gamma-1}\rho_{b}^{\gamma-1}\widehat{A(s)}+\frac{2\mu u^0}{\rho_b} \hat{p}+H,
\end{align}
where
\begin{eqnarray}\label{eq323}
H=\mu u^0\Big((u^1)^2+(u^2)^2+O(1)\frac{2}{\gamma-1}(|\hat{p}|^2+|\widehat{A(s)}|^2)\Big),
\end{eqnarray}
and $O(1)$ represents a bounded quantity depending only on the background solution.

\subsubsection{Linearization of pressure's equation}\label{sec242}
By setting $t={u_b^2}/{c_b^2}=M_b^2\in (0,1),$
direct computation yields that \eqref{eq306} can be written as
\begin{align}\label{eq324}
\mathcal{L}(\hat{p})\triangleq&(t-1)\p_0^2\hat{p}-\p_1^2\hat{p}-\p_2^2\hat{p}+\mu
d_1(t)\p_0\hat{p}+\mu^2d_2(t)\hat{p}+\mu^2\rho_bd_3(t)\hat{E}
+\mu^2\rho_b^\gamma d_4(t)\widehat{A(s)}\nonumber\\
=&F_5\triangleq \frac{1}{c_b^2}(F_3+F_4).
\end{align}
We easily see that \eqref{eq324} is an elliptic equation of (perturbed) pressure for subsonic flow. The coefficients in \eqref{eq324} are given by
\begin{align*}
d_1(t)\triangleq&
-\frac{1}{t-1}\left((1+2\gamma)t^2+t-2\right),\\
d_2(t)\triangleq&\frac{1}{(t-1)^3}\left(\gamma(1+\gamma)t^4-
2\gamma(1+\gamma)t^3-(\gamma-3)t^2-2(4+\gamma)t+8\right), \\
d_3(t)\triangleq&\frac{1}{(t-1)^3}\Big(\gamma t^2+3t-4\Big),\\
d_4(t)\triangleq&-\frac{1}{\gamma-1}\frac{1}{(t-1)^3}\left(\gamma(\gamma-1)t^3+(5\gamma-3)
t^2-2(\gamma-4)t-8\right),
\end{align*}
and
\begin{align}\label{eq329}
-F_4
=&\left(2\hat{E}-\frac{\gamma+1}{\gamma-1}(c^2-c_b^2)\right)\p_0^2\hat{p}
-(c^2-c_b^2)(\p_1^2\hat{p}+\p_2^2\hat{p})
+2\mu^2\gamma\hat{p}\left(\hat{E}-\frac{1}{\gamma-1}
(c^2-c_b^2)\right)\nonumber\\
&\quad-2\mu\p_0\hat{p}\left(\hat{E}-\frac{1}{\gamma-1}(c^2-c_b^2)\right)
-\p_0(p+p_b)\p_0\hat{p}\left(\frac{|u|^2}{p}-\frac{u_b^2}{p_b}\right)
-\frac{u_b^2}{p_b}(\p_0\hat{p})^2\nonumber\\
&\quad\quad+(\p_0p_b)^2\frac{\hat{p}}{p_b}\left(\frac{|u|^2}{p}-\frac{u_b^2}{p_b}\right)
-\frac{\p_0(p+p_b)\p_0\hat{p}}{\gamma}\left(\frac{c^4}{p|u|^2}-\frac{c_b^4}{p_bu_b^2}\right)
-\frac{c_b^4}{\gamma p_bu_b^2}(\p_0\hat{p})^2\nonumber\\
&\quad\quad\quad-\frac{(\p_0p_b)^2}{\gamma}\left(\frac{1}{|u|^2}-\frac{1}{u_b^2}\right)
\left\{\left(\frac{c^4}{p}-\frac{c_b^4}{p_b}\right)+\frac{2c_b^4}{p_bu_b^2}
\left(\frac{c^2-c_b^2}{\gamma-1}-\hat{E}\right)\right\}\nonumber\\
&\quad\quad\quad\quad-\frac{(\p_0p_b)^2}{\gamma u_b^2}\left\{\left(\frac{c^2+c_b^2}{p}-\frac{2c_b^2}{p_b}\right)(c^2-c_b^2)
-\frac{c_b^4}{p_b}\hat{p}\left(\frac{1}{p}-\frac{1}{p_b}\right)\right\}
\nonumber\\
&\quad\quad\quad\quad\quad\quad-O(1)\Big(|\hat{p}|^2+|\widehat{A(s)}|^2\Big)
\left\{\frac{\gamma+1}{\gamma-1}\p_0^2p_b
+\frac{2\mu^2\gamma}{\gamma-1}p_b-\frac{2\mu}{\gamma-1}\p_0p_b\right.\nonumber\\
&\quad\quad\quad\quad\quad\quad\quad\left.-\frac{2}{\gamma-1}\frac{(\p_0p_b)^2}{p_b}
+\frac{2c_b^2}{\gamma p_b u_b^2}(\p_0p_b)^2\left(1+\frac{1}{\gamma-1}\frac{c_b^2}{u_b^2}\right)\right\}.
\end{align}

\subsubsection{Linearization of boundary condition}\label{sec243}
Note that \eqref{eq309} is equivalent to
\begin{eqnarray}\label{eq330}
\p_0(p-p_b)-\mu\gamma\left(\frac{ p
	(u^0)^2}{(u^0)^2-c^2}-\frac{p_b
	(u_b)^2}{(u_b)^2-c_b^2}\right)=G_1+G_2,
\end{eqnarray}
which could  be written as
\begin{eqnarray}\label{eq331}
\p_0\hat{p}+\gamma_0\hat{p}=G\triangleq G_1+G_2+G_3,
\end{eqnarray}
with $\gamma_0$  determined by the background solution:
\begin{eqnarray}\label{eq332}
\gamma_0\triangleq-\mu\frac{\gamma t^2-t+2}{(1-t)^2}<0,
\end{eqnarray}
and
\begin{align}\label{eq333}
G_3&=\mu\gamma\left\{\left(\frac{(u^0)^2}{(u^0)^2-c^2}-\frac{u_b^2}{u_b^2-c_b^2}\right)\hat{p}
+p_b\frac{u_b^2(c^2-c_b^2)-c_b^2((u^0)^2-u_b^2)}{u_b^2-c_b^2}\right.\nonumber\\
&\qquad\left.\times\left(\frac{1}{(u^0)^2-c^2}-\frac{1}{u_b^2-c_b^2}\right)
-\frac{p_bc_b^2}{(u_b^2-c_b^2)^2}\Big(2\hat{E}-
(u^1)^2-(u^2)^2\Big)\right.\nonumber\\
&\qquad\qquad\left.+\frac{2p_bE_b}{(u_b^2-c_b^2)^2}\left(O(1)(|\hat{p}|^2+|\widehat{A(s)}|^2)
+\rho_b^{\gamma-1}\widehat{A(s)}\right)\right\}.
\end{align}

\subsubsection{Linearization of R-H conditions}\label{sec244}
Next, we linearize the R--H conditions \eqref{eq312} and \eqref{eq313}. Let $V=(u^0,p,\rho)^\top,$ and  $V^-=((u^0)^-,p^-,\rho^-)^\top.$ Then we write them  equivalently as
\begin{eqnarray}\label{eq335}
\mathcal{G}_i(V, V^-)=\Psi_i(U, U^-, D\psi),\qquad i=1,2,3,
\end{eqnarray}
with
\begin{align*}
& \mathcal{G}_1=[ \rho(u^0)^2+p], & \Psi_1&=\p_{1}\psi[ \rho u^0 u^1]+\p_{2}\psi[ \rho u^0 u^2],\\
& \mathcal{G}_2=[ \rho u^0], & \Psi_2&=\p_{1}\psi[\rho u^1]+\p_{2}\psi[\rho u^2],\\
& \mathcal{G}_3=[ E], & \Psi_3&=0.
\end{align*}

As in \cite[p.2522]{ChenYuan2013}, using the fact that $$\mathcal{G}_i(V_b^+(r_b,x'),V_b^-(r_b,x'))=0,\quad i=1,2,3,$$ for $V_b^+=((u^0)_b^+,p_b^+,\rho_b^+)^\top$ and  $V_b^-=((u^0)_b^-,p_b^-,\rho_b^-)^\top,$  we have
\begin{align}\label{eq337}
&\p_+\mathcal{G}_i(V_b^+(r_b,x'),V_b^-(r_b,x'))\bullet\Big(V(\psi(x'),x')
-V_b^+(\psi(x'),x')\Big)\\
&=\Big\{-\Big(\p_+\mathcal{G}_i(V_b^+(\psi(x'),x'),V_b^-(\psi(x'),x'))
-\p_+\mathcal{G}_i(V_b^+(r_b,x'),V_b^-(r_b,x'))\Big)\nonumber\\
&\quad\bullet\Big(V(\psi(x'),x')-V_b^+(\psi(x'),x')\Big)\nonumber\\
&\quad+\,\,\p_+\mathcal{G}_i(V_b^+(\psi(x'),x'),V_b^-(\psi(x'),x'))\bullet\Big(V(\psi(x'),
x')-V_b^+(\psi(x'),x')\Big)\nonumber\\
&\quad-\Big(\mathcal{G}_i(V(\psi(x'),x'),V^-(\psi(x'),x'))-\mathcal{G}_i(V_b^+(\psi(x'),x'),
V^-(\psi(x'),x'))\Big)\nonumber\\
&\quad-\Big(\mathcal{G}_i(V_b^+(\psi(x'),x'),V^-(\psi(x'),x'))-\mathcal{G}_i(V_b^+(\psi(x'),
x'),V_b^-(\psi(x'),x'))\Big) +\Psi_i\Big\}\nonumber\\
&\quad-\Big\{\mathcal{G}_i(V_b^+(\psi(x'),x'),V_b^-(\psi(x'),x'))-\mathcal{G}_i(V_b^+(r_b,x'),
V_b^-(r_b,x'))\Big\}\nonumber\\
&\triangleq\Rmnum{1}_i + \Rmnum{2}_i,\nonumber
\end{align}
where ``$\bullet$" denotes the scalar product of vectors in the phase (Euclidean) space $\mathbb{R}^3$, and $\p_+\mathcal{G}_i(V,V^-)$ (respectively $\p_-\mathcal{G}_i(V,V^-)$) is the gradient of $\mathcal{G}_i(V,V^-)$ with respect to the variables $V$ (respectively $V^-$).

By the Taylor expansion formula, all the five terms in ${\Rmnum{1}_i}$ are of higher-order for $i=1,2,3$. While it is not the case in
\begin{align*}
\Rmnum{2}_i=&\p_+\mathcal{G}_i(V_b^+(r_b,x'),V_b^-(r_b,x'))\frac{\dd V_b^+}{\dd x^0}(r_b)(\psi^p
+r^p-r_b)\\
&\quad+\,\p_-\mathcal{G}_i(V_b^+(r_b,x'),V_b^-(r_b,x'))\frac{\dd V_b^-}{\dd x^0}(r_b)
(\psi^p+r^p-r_b)+O(1)|\psi-r_b|^2
\end{align*}
for $i=1,3.$ We remark that this is quite essential for stabilization of shocks. Actually, using \eqref{eqback} and the result
\begin{align*}
&\det\Big(\left.\frac{\p(\mathcal{G}_1,\mathcal{G}_2,\mathcal{G}_3)}
{\p(u^0,p,\rho)}\Big)\right|_{{(V_b^-,V_b^+; r_b)}}\nonumber\\
=&\det \left.\left(
  \begin{array}{ccc}
    2\rho_b^+ (u^0)_b^+ & 1 & ((u^0)_b^+)^2 \\
    \rho_b^+ & 0 & (u^0)_b^+ \\
    (u^0)_b^+ & \frac{\gamma}{\gamma-1}\frac{1}{\rho_b^+}
    & -\frac{1}{\gamma-1}\frac{(c^2)_b^+}{\rho_b^+} \\
  \end{array}
\right)\right|_{x^0=r_b}
=\frac{(c^2-(u^0)^2)_b^+(r_b)}{\gamma-1}>0
\end{align*}
obtained by direct calculations,
\eqref{eq337} equals to
\begin{eqnarray*}
&&\left.\left(
  \begin{array}{ccc}
    2\rho_b^+ (u^0)_b^+ & 1 & ((u^0)_b^+)^2 \\
    \rho_b^+ & 0 & (u^0)_b^+ \\
    (u^0)_b^+ & \frac{\gamma}{\gamma-1}\frac{1}{\rho_b^+}
    & -\frac{1}{\gamma-1}\frac{(c^2)_b^+}{\rho_b^+} \\
  \end{array}
\right)\right|_{x^0=r_b}
\left.\left(  \begin{array}{c}
     (\widehat{u^0})\\
    \hat{p} \\
    \hat{\rho} \\
  \end{array}
\right)\right|_{S^\psi}\nonumber\\
&=&
\left(
  \begin{array}{c}
     -\mu(p_b^+-p_b^-)(r_b)\\
    0 \\
    -\frac{2\mu }{\gamma-1}\Big((c^2)_b^+-(c^2)_b^-\Big)(r_b)\\
  \end{array}
\right)(\psi^p+r^p-r_b)+\text{\it higher-order terms}.
\end{eqnarray*}

We can solve these linear algebraic equations to get
\begin{align}
\widehat{u^0}|_{S^\psi}=&\mu_1\,(\psi^p+r^p-r_b)+g_1(U,U^-,\psi,D\psi),
\label{eq338}\\
\hat{p}|_{S^\psi}=&\mu_2\,(\psi^p+r^p-r_b)+g_2(U,U^-,\psi,D\psi),
\label{eq339}\\
\hat{\rho}|_{S^\psi}=&\mu_3\,(\psi^p+r^p-r_b)+g_3(U,U^-,\psi,D\psi).
\label{eq340}
\end{align}
Using $A(S)=p\rho^{-\gamma}$, it also holds that
\begin{eqnarray}
\widehat{A(S)}|_{S^\psi}=\mu_4\,(\psi^p+r^p-r_b)+g_4(U,U^-,\psi,D\psi),\label{eq341}
\end{eqnarray}
where
\begin{align*}
\mu_1=&2\mu(u^0)^+_b(r_b)\left\{\frac{\gamma}{\gamma+1}+\frac{1}{(c^2-(u^0)^2)_b^+(r_b)}
\Big((c^2)_b^--(c^2)_b^+\Big)(r_b)\right\}>0,\\
\mu_2=&-2\mu\rho_b^+(r_b)\left\{\frac{1}{\gamma+1}\Big((\gamma-1)(u^0)^2+c^2\Big)_b^+(r_b)
+\frac{((u^0)^+_b)^2(r_b)}{(c^2-(u^0)^2)_b^+(r_b)}\Big((c^2)_b^--(c^2)_b^+\Big)(r_b)\right\}<0,\\
\mu_3=&-2\mu\rho_b^+(r_b)\left\{\frac{\gamma}{\gamma+1}+\frac{1}{(c^2-(u^0)^2)_b^+(r_b)}
\Big((c^2)_b^--(c^2)_b^+\Big)(r_b)\right\}<0,\\
\mu_4=&2\mu(\rho_b^+(r_b))^{1-\gamma}\left\{\frac{\gamma-1}{\gamma+1}\Big(c^2-(u^0)^2\Big)_b^+
(r_b)+\Big((c^2)_b^--(c^2)_b^+\Big)(r_b)\right\}>0,
\end{align*}
and $g_k (k=1,2,3,4)$ are higher-order terms.  Observing that if the friction disappears, namely $\mu=0$, then all the coefficients above are zero, and there is no couplings in equations \eqref{eq338}-\eqref{eq340}  on the position of shock-front. This is one of the key point why friction may have a stabilization effect on transonic shocks.

From \eqref{eq339} and \eqref{eq341}, we also have
\begin{eqnarray}\label{eq344s}
\widehat{A(s)}|_{S^\psi}=\frac{\mu_4}{\mu_2}\hat{p}|_{S^\psi}+g_4-\frac{\mu_4}{\mu_2}g_2.
\end{eqnarray}

\subsubsection{Divergence of tangential velocity field on shock-front}\label{sec245}
Now we restrict \eqref{eq331} on $S^\psi$. So particularly $x^0$ should be replaced by $\psi$. Using the commutator relation for a function $f$ defined on $\Omega$:
\begin{eqnarray*}
\p_\beta(f|_{S^\psi})=\p_\beta\psi(\p_0f)|_{S^\psi}+(\p_\beta f)|_{S^\psi}, \quad \beta=1,2,
\end{eqnarray*}
we have
\begin{eqnarray}
G_1|_{S^\psi}
=\left.\left(\frac{-\rho u^0}{\left(\frac{u^0}{c}\right)^2-1}\right)\right|_{S^\psi}
\left(\p_\beta(u^\beta|_{S^\psi})-\p_\beta\psi(\p_0u^\beta)|_{S^\psi}\right),\label{eq342}
\end{eqnarray}
and
\begin{align*}
\Big\{\p_\beta\psi(u^j\p_ju^\beta+\frac{1}{\rho}\p_\beta p)\Big\}\Big|_{S^\psi}
&=\left.\left\{(u^0-u^\delta\p_\delta\psi)\p_\beta\psi \p_0u^\beta+\frac{1}{\rho}\p_\beta\psi\p_\beta p\right\}\right|_{S^\psi}\nonumber\\
&\qquad+\, \p_\beta\psi(u^\sigma|_{S^\psi})\p_\sigma(u^\beta|_{S^\psi}).
\end{align*}

One then solves $(\p_\beta\psi \p_0u^\beta)|_{S^\psi}$, and  \eqref{eq342} becomes
\begin{align*}
G_1|_{S^\psi}&=\left.\left(\frac{-\rho u^0}{\left(\frac{u^0}{c}\right)^2-1}\right)\right|_{S^\psi}
\p_\beta(u^\beta|_{S^\psi})+\left.\left(\frac{\rho u^0}{\left(\frac{u^0}{c}\right)^2-1}\frac{1}{u^0-u^\delta\p_\delta\psi}\right)\right|_{S^\psi}\nonumber\\
&\quad\times\Big\{
\p_\beta\psi(D_uu^\beta+\frac{1}{\rho}\p_\beta p)-\frac{1}{\rho}\p_\beta\psi\p_\beta p-\p_\beta\psi u^\sigma \p_\sigma(u^\beta|_{S^\psi})\Big\}\Big|_{S^\psi}.
\end{align*}
Similarly, one may replace the normal derivatives by tangential derivatives to compute $G_2|_{S^\psi}$ and $G_3|_{S^\psi}$.
Then \eqref{eq331} becomes
\begin{eqnarray}\label{eq343p}
\left.\left\{\p_0\hat{p}+\gamma_1\hat{p}+\gamma_2\p_\beta(u^\beta|_{S^\psi})\right\}\right|_{S^\psi}
+{G}_4=0.
\end{eqnarray}
Here $\gamma_1<0, \gamma_2<0$  are constants determined by the background solution, and
\begin{align}\label{eq343}
{G}_4&\triangleq\left.\left(\frac{\rho u^0}{\left(\frac{u^0}{c}\right)^2-1}-\gamma_2\right)\right|_{S^\psi}\p_\beta(u^\beta|_{S^\psi})
-\mu\left\{\left.\left(\frac{\gamma t^2-t+2}{(1-t)^2}\right)\right|_{S^\psi}-\left.\frac{\gamma t^2-t+2}{(1-t)^2}\right|_{x^0=r_b}\right\}(\hat{p}|_{S^\psi})\nonumber\\
&+\left.\left\{\frac{ u^0}{\left(\frac{u^0}{c}\right)^2-1}\left(\frac{\p_\beta\psi }{u^0-u^\delta\p_\delta\psi}
\left(\p_\beta p+\rho u^\sigma\p_\sigma(u^\beta|_{S^\psi})\right)+\left(\frac{1}{c^2}+\frac{1}{(u^0)^2}\right)u^\beta\p_\beta p
\right)\right\}\right|_{S^\psi}\nonumber\\
&+\left.\left\{\frac{1}{\left(\frac{u^0}{c}\right)^2-1}\frac{u^\beta}{u^0-u^\delta\p_\delta\psi}
\left(\frac{\rho^\gamma}{\gamma-1}\p_\beta(A(s)|_{S^\psi})-\rho\p_\beta(E|_{S^\psi})-\mu\rho (u^0)^2\p_\beta\psi\right)\right\}\right|_{S^\psi}\nonumber\\
&+\left.\left\{\frac{1}{\left(\frac{u^0}{c}\right)^2-1}
\frac{1}{u^0-u^\delta\p_\delta\psi} \left( \rho u^\beta u^\sigma\p_\sigma(u^\beta|_{S^\psi})+\frac{1}{u^0}u^\beta u^\sigma\p_\sigma\psi \p_\beta p\right)\right\}\right|_{S^\psi}\nonumber\\
&+\, O(1)\left.\Big\{|\hat{U}|^2+|{\psi}-r_b|^2+\hat{E}\Big\}\right|_{S^\psi}-\mu
\left.\left\{\frac{t+\frac{2}{\gamma-1}}{(1-t)^2}\rho_b^\gamma \left(g_4-\frac{\mu_4}{\mu_2}g_2\right)\right\}\right|_{S^\psi}.
\end{align}

We see that \eqref{eq343p} is equivalent to
\begin{eqnarray}\label{eq344}
\p_\beta(u^\beta|_{S^\psi})=\mu_5\, (\p_0\hat{p})|_{S^\psi}+\mu_6\,
\psi^p+\mu_6\,(r^p-r_b)+g_5(U,U^-,\psi,DU,D\psi),
\end{eqnarray}
if we replace $\hat{p}|_{S^\psi}$ by ${\psi}$ via \eqref{eq339}. Here
$\mu_5>0,  \mu_6>0$ are constants determined by the background solution, and
\begin{eqnarray*}
g_5\triangleq-\frac{1}{\gamma_2}{G}_4-\frac{\gamma_1}{\gamma_2}g_2.
\end{eqnarray*}
\begin{remark}\label{rem33}
As observed in \cite[p.728]{LiuXuYuan2016}, in the expression of $g_5$, there appear first-order derivatives of $\hat{p}$, and  only first-order tangential derivatives of $A(s), u', E, \psi$ on $S^\psi$.
Also \eqref{eq344} is a first-order boundary condition on the shock-front. Together with \eqref{eq318}, we have a div-curl system of  the tangential velocity  ${u}'|_{S^\psi}$ on $\T^2$.
\end{remark}

\subsubsection{Problem (T1)}\label{sec246}
For functions $U=(E, A(s), p, u')$ and $\psi$ (note that $\hat{U}=U-U_b^+$ and $\psi=\psi^p+r^p$), we formulate the following problems:
\begin{eqnarray}
&&\begin{cases}\label{eq346}
 D_u \hat{E}+2\mu u^0 \hat{E}=\frac{2\mu u^0 }{\gamma-1}\rho_{b}^{\gamma-1}\widehat{A(s)}+\frac{2\mu u^0}{\rho_b} \hat{p}+H(U)&\text{in}\quad \Omega^+_\psi,\\[3pt]
E=E^-&\text{on}\quad S^\psi;
\end{cases}\\
&&\begin{cases}\label{eq347}
\mathcal{L}(\hat{p})=F_5(U, DU, D^2p)& \text{in}\quad \Omega^+_\psi,\\[3pt]
\hat{p}=p_1-p_b^+ &\text{on}\quad \Sigma_1,\\[3pt]
\hat{p}=\mu_2\,(\psi^p+r^p-r_b)+g_2(U,U^-,\psi,D\psi)&\text{on}\quad S^\psi;
\end{cases}\\
&&\begin{cases}\label{eq348}
D_uA(s)=0&\text{in}\quad\Omega^+_\psi,\\[3pt]
\widehat{A(s)}=\mu_4\,(\psi^p+r^p-r_b)+g_4(U,U^-,\psi,D\psi)&\text{on}\quad S^\psi;
\end{cases}\\
&&\begin{cases}\label{eq349}
\p_\beta\psi=\mu_0({u}^\beta|_{S^\psi})+g_0^\beta(U,U^-,D\psi),\qquad \beta=1,2\quad\text{on} \ \  \T^2,\\[3pt]
\int_{0}^{2\pi}\Big(\mu_0({u}^1|_{S^\psi})+g_0^1\Big)(\psi(s,\pi),s,\pi)ds=0,\\
\int_{0}^{2\pi}\Big(\mu_0({u}^2|_{S^\psi})+g_0^2\Big)(\psi(\pi,s),\pi,s)ds=0,\\
\p_\beta(u^\beta|_{S^\psi})=\mu_5\,
(\p_0\hat{p}|_{S^\psi})+\mu_6\,
\psi^p+\mu_6\,(r^p-r_b)+g_5(U,U^-,\psi,DU,D\psi);
\end{cases}\\
&&\begin{cases}\label{eq350}
D_uu^\beta=-\frac{1}{\rho}\p_\beta p&\text{in}\quad \Omega^+_\psi,\quad \beta=1,2,\\[3pt]
u^\beta=u^\beta_0&\text{on}\quad S^\psi.
\end{cases}
\end{eqnarray}
The initial data $u^\beta_0$ in \eqref{eq350} is obtained from the vector field $(u^1, u^2)$ on $\T^2$ defined  in \eqref{eq349}.

It is obvious now that for given supersonic flow $U^-$, the solution $(U, \psi)$ to these problems also solves the Euler system,  and the R-H conditions hold across $S^\psi.$ Hence  we could rewrite Problem (T) equivalently as the following Problem (T1).

\medskip
\fbox{
\parbox{0.90\textwidth}{
Problem (T1): Find $\psi$ and $U=U^+$ in $\Omega_\psi^+$ satisfying \eqref{eq211}, \eqref{eq213} and solving the problems \eqref{eq346}--\eqref{eq350}.}}
\medskip

\subsection{Problem (T2)}\label{sec25}

Acting the divergence operator to the first equation in \eqref{eq349} and using the second equation, we derive that
\begin{align}\label{eq351}
\Delta'\psi^p+\mu_7\psi^p=&\mu_0\mu_6(r^p-r_b)+\mu_0\mu_5\,(\p_0\hat{p}|_{S^\psi})\nonumber\\
&\quad+g_6(U, U^-, \psi, DU^-, DU, D\psi, D^2\psi),
\end{align}
with $g_6=\mu_0g_5+\p_\beta g_0^\beta$ and $\mu_7=-\mu_0\mu_6<0$. Here $\Delta'=\sum_{\beta=1}^2\p^2_{\beta}$ is the standard Laplace operator on $\T^2$.

Then using the third equation in \eqref{eq347}, we get
\begin{align}\label{eq352}
&\Delta'(\hat{p}|_{S^\psi})+\mu_7({\hat{p}}|_{S^\psi})+\mu_8(\p_0\hat{p}|_{S^\psi})\nonumber\\
&\qquad\qquad=g_8(U, U^-, \psi, DU, DU^-, D\psi, D^2U, D^2U^-, D^2\psi, D^3\psi),
\end{align}
where $\mu_8=-\mu_0\mu_2\mu_5>0$ and $g_8=\Delta'g_2+\mu_7g_2+\mu_2g_6.$

By the last equation in \eqref{eq349}, using the divergence theorem, and recall that $\int_{\T^2}\psi^p\dd x^1\dd x^2=0,$ we have
\begin{eqnarray}\label{eq353}
r^p-r_b=-\frac{1}{4\pi^2\mu_6}\int_{{\T}^2}\Big(\mu_5\,
(\p_0\hat{p}|_{S^\psi})+g_5(U,U^-,\psi,DU,D\psi)\Big)\, \dd x^1\dd x^2.
\end{eqnarray}
Substituting this into the third equation in \eqref{eq347}, we then obtain
\begin{eqnarray}\label{eq354}
\psi^p=\frac{1}{\mu_2}\left((\hat{p}|_{S^\psi})-{\mu_9}
\int_{{\T}^2}(\p_0\hat{p}|_{S^\psi})\,\dd x^1\dd x^2+{g_9(U,U^-,\psi,D\psi)}\right),
\label{177}
\end{eqnarray}
with $\mu_9=-\frac{\mu_2\mu_5}{4\pi^2\mu_6}>0$ and
$g_9=\frac{\mu_2}{4\pi^2\mu_6}\int_{{\T}^2}g_5\,\dd x^1\dd x^2-g_2.$

We now formulate the following Problem (T2), which  is equivalent to Problem (T1) ({\it cf.} \cite[p.730]{LiuXuYuan2016}).

\medskip
\fbox{
\parbox{0.90\textwidth}{
Problem (T2): Find $\psi$ and $U=\hat{U}+U_b^+$ that solve  \eqref{eq346}, \eqref{eq355}, \eqref{eq356}, \eqref{eq348}, \eqref{eq357} and \eqref{eq350}.}}

\begin{eqnarray}
&&\begin{cases}\label{eq355}
\mathcal{L}(\hat{p})=F_5(U, DU, D^2p)& \text{in}\quad \Omega^+_\psi,\\[3pt]
\hat{p}=p_1-p_b^+ &\text{on}\quad \Sigma_1,\\[3pt]
\Delta'(\hat{p}|_{S^\psi})+\mu_7({\hat{p}}|_{S^\psi})+\mu_8(\p_0\hat{p}|_{S^\psi})\\[3pt]
\qquad=g_8(U, U^-, \psi, DU, DU^-, D\psi, D^2U, D^2U^-, D^2\psi, D^3\psi)&\text{on}\quad S^\psi;
\end{cases}\\
&&\begin{cases}\label{eq356}
r^p-r_b=-\frac{1}{4\pi^2\mu_6}\int_{{\T}^2}\Big(\mu_5\,
(\p_0\hat{p}|_{S^\psi})+g_5(U,U^-,\psi,DU,D\psi)\Big)\, \dd x^1\dd x^2,\\[3pt]
\psi^p=\frac{1}{\mu_2}\Big((\hat{p}|_{S^\psi})-{\mu_9}
\int_{{\T}^2}(\p_0\hat{p}|_{S^\psi})\,\dd x^1\dd x^2+{g_9(U,U^-,\psi,D\psi)}\Big),\\[3pt]
\psi=\psi^p+r^p;
\end{cases}\\
&&\begin{cases}\label{eq357}
\p_2({u}^1|_{S^\psi})-\p_1({u}^2|_{S^\psi})=-\frac{1}{\mu_0}(\p_2g_0^1-\p_1g_0^2),\\[3pt]
\p_\beta(u^\beta|_{S^\psi})=\mu_5\,
(\p_0\hat{p}|_{S^\psi})+\mu_6\,
\psi^p+\mu_6\,(r^p-r_b)+g_5(U,U^-,\psi,DU,D\psi),\\[3pt]
\int_{0}^{2\pi}\Big(\mu_0({u}^1|_{S^\psi})+g_0^1\Big)(\psi(s,x^2),s,x^2)\,\dd s=0,\\[3pt]
\int_{0}^{2\pi}\Big(\mu_0({u}^2|_{S^\psi})+g_0^2\Big)(\psi(x^1,s),x^1,s)\,\dd s=0 \quad \text{on}\, \T^2.
\end{cases}
\end{eqnarray}

\subsection{Problem (T3)}\label{sec26}
The above equations and boundary conditions are formulated in $\Omega_\psi^+.$
Supposing $\psi\in C^{4,\alpha}(\T^2)$, we introduce a $C^{4,\alpha}$--homeomorphism $ \Psi:
(x^0,x')\in\Omega_\psi^+\mapsto ({y}^0,y')\in\mathcal{M}\triangleq(r_b,L)\times {\T}^2$ defined by
\begin{eqnarray}\label{eq362}
{y}^0=\frac{x^0-L}{L-\psi(x')}(L-r_b)+L,\qquad y'=(y^1, y^2)=x'=(x^1, x^2)
\end{eqnarray}
to normalize $\Omega_\psi^+$ to $\mathcal{M}$. Then
$\p\mathcal{M}=\mathcal{M}_0\cup\mathcal{M}_1$ with $\mathcal{M}_{0}=\{r_b\}\times{\T}^2$ and $\mathcal{M}_{1}=\{L\}\times{\T}^2$. They are respectively the images of $S^\psi$ and $\Sigma_1$.

To avoid complication of notations, {\it in the following we still write the unknowns $U(\Psi^{-1}(y))$  in $y$-coordinates as $U$ etc,} and write the velocity
\begin{eqnarray}\label{eq363}
\left(\frac{L-r_b}{L-\psi(y')}u^0(\Psi^{-1}(y))+\frac{y^0-L}{L-\psi(y')}u^\beta\p_\beta \psi(y'), u^1(\Psi^{-1}(y)), u^2(\Psi^{-1}(y))\right)^\top
\end{eqnarray}
still as $u$. We have
\begin{align*}
\Delta'_x\hat{p}&=\Delta'_y\hat{p}+O(1)\Big(D^2\hat{p}D\psi
+D\hat{p}D^2\psi+D\hat{p}D\psi\Big),\\
\left.\left({\frac{\p \hat{p}}{\p{x^0}}}\right)\right|_{S^\psi}&=i^*\left(\frac{\p \hat{p}}{\p{y^0}}\right)+O(1)(\psi-r_b)D\hat{p}.
\end{align*}
Here we use $i^*$ to denote the trace operator on $\mathcal{M}_0$.

Hence Problem (T2) could be rewritten as the following Problem (T3) in the $y$-coordinates, where we use $\overline{F}$ or $\bar{g}$ to denote the corresponding higher-order terms.

\medskip
\fbox{
\parbox{0.90\textwidth}{
Problem (T3): Find $\psi\in C^{4,\alpha}(\T^2)$ and $U=U_b^++\hat{U}$ that solve the following problems \eqref{eq364}--\eqref{eq369}. The initial data $u_0'$ in \eqref{eq369}  is the vector field corresponding to $\bar{u}_0'$ on $\T^2$ obtained from \eqref{eq368}.}}

\begin{eqnarray}
&&\begin{cases}\label{eq364}
D_u \hat{E}+2\mu u^0 \hat{E}=\frac{2\mu u^0 }{\gamma-1}\rho_{b}^{\gamma-1}\widehat{A(s)}+\frac{2\mu u^0}{\rho_b} \hat{p}+\overline{H}(U, \psi, D\psi)&\text{in}\quad \mathcal{M},\\[3pt]
\hat{E}=E^--E_b^-&\text{on}\quad \mathcal{M}_0;
\end{cases}\\
&&\begin{cases}\label{eq365}
\mathcal{L}(\hat{p})=\overline{F}_5(U, \psi, DU, D\psi, D^2p, D^2\psi)& \text{in}\quad \mathcal{M},\\[3pt]
\hat{p}=p_1-p_b^+ &\text{on}\quad \mathcal{M}_1,\\[3pt]
\Delta'(i^*\hat{p})+\mu_7(i^*{\hat{p}})+\mu_8(i^*\p_0\hat{p})\\[3pt]
\qquad=\bar{g}_8(U,U^-,\psi,DU,D\psi,D^2U,D^2\psi,D^3\psi)  &\text{on}\quad \mathcal{M}_0;
\end{cases}\\
&&\begin{cases}\label{eq366}
r^p-r_b=-\frac{1}{4\pi^2\mu_6}\int_{{\T}^2}\Big(\mu_5\,
i^*(\p_0\hat{p})+\bar{g}_5(U,U^-,\psi,DU,D\psi)\Big)\, \dd x^1\dd x^2,\\[3pt]
\psi^p=\frac{1}{\mu_2}\Big(i^*(\hat{p})-{\mu_9}
\int_{{\T}^2}i^*(\p_0\hat{p})\,\dd x^1\dd x^2+{\bar{g}_9(U,U^-,\psi,D\psi)}\Big),\\[3pt]
\psi=\psi^p+r^p;
\end{cases}\\
&&\begin{cases}\label{eq367}
D_uA(s)=0&\text{in}\quad \mathcal{M},\\[3pt]
i^*(\widehat{A(s)})=\mu_4\,(\psi^p+r^p-r_b)+\bar{g}_4(U,U^-,\psi,D\psi)&\text{on}\quad \mathcal{M}_0;
\end{cases}\\
&&\begin{cases}\label{eq368}
\p_2({\bar{u}}^1_0)-\p_1({\bar{u}}^2_0)=-\frac{1}{\mu_0}\Big(\p_2\bar{g}_0^1(U, U^-,D\psi)-\p_1\bar{g}_0^2(U, U^-,D\psi)\Big),\\[3pt]
\p_\beta(\bar{u}^\beta_0)=\mu_5\,
i^*(\p_0\hat{p})+\mu_6\,\psi^p+\mu_6\,(r^p-r_b)+\bar{g}_5(U,U^-,\psi,DU,D\psi),\\[3pt]
\int_{0}^{2\pi}\Big(\mu_0\bar{u}^1_0+\bar{g}_0^1\Big)(s,x^2)\,\dd s=0,\\[3pt]
\int_{0}^{2\pi}\Big(\mu_0\bar{u}^2_0+\bar{g}_0^2\Big)(x^1,s)\,\dd s=0\qquad\text{on}\, \T^2;
\end{cases}
\end{eqnarray}
\begin{eqnarray}
&&\begin{cases}\label{eq369}
D_uu^\beta=-\frac{1}{\rho}\p_\beta \hat{p}+\overline{W}_{\beta}(U,\psi,Dp, D\psi)&\text{in}\quad \mathcal{M},\quad \beta=1,2,\\[3pt]
u^\beta=u^\beta_0&\text{on}\quad \mathcal{M}_0.
\end{cases}
\end{eqnarray}

\begin{remark}
Recalling \eqref{eq324}, in \eqref{eq365} we should have
\begin{align}\label{eq370}
\mathcal{L}(\hat{p})&\triangleq
(t(y^0)-1)\p_0^2\hat{p}-\p_1^2\hat{p}-\p_2^2\hat{p}+\mu
d_1(t(y^0))\p_0\hat{p}+\mu^2d_2(t(y^0))\hat{p}\nonumber\\
&\qquad+\, \mu^2\rho_b(y^0)d_3(t(y^0))\hat{E}
+\mu^2\rho_b(y^0)^\gamma d_4(t(y^0))\widehat{A(s)}\nonumber\\
&=\overline{F}_5\triangleq F_5+{F}_6,
\end{align}
where the coefficients are known functions of $y^0$, and $F_5=F_5(U,\psi, DU,D^2p, D\psi, D^2\psi)$ is the higher-order term appeared below \eqref{eq324} in the $y$-coordinates, and
\begin{eqnarray}\label{eq371}
{F}_6=O(1)\Big(D^2\hat{p}D\psi+D\hat{p}D^2\psi+D\hat{p}D\psi\Big).
\end{eqnarray}
We also note that
\begin{eqnarray*}
&&\bar{H}=H+O(1)(\psi-r_b)\hat{U}+O(1)(\hat{U}D\psi),\quad
\bar{g}_k=g_k+O(1)(\psi-r_b)D\hat{p},\quad k=5,8,9,\\
&& \bar{g}_4=g_4, \quad \bar{g}^\beta_0=g^\beta_0, \quad \overline{W}_{\beta}=-\frac{1}{\rho}\left(\frac{y^0-L}{L-\psi(y')}\p_\beta\psi\right)\p_0 \hat{p},\quad\beta=1,2.
\end{eqnarray*}
\end{remark}

\subsection{Problem (T4)}\label{sec27}
Since the elliptic problem \eqref{eq365} is coupled with the other hyperbolic problems, we need to further reformulate Problem (T3) equivalently as the following Problem (T4).

We now consider the Cauchy problems \eqref{eq364} and  \eqref{eq367}:
\begin{eqnarray}
&&\begin{cases}\label{eq378E}
D_u \hat{E}+2\mu u^0 \hat{E}=\frac{2\mu u^0 }{\gamma-1}\rho_{b}^{\gamma-1}\widehat{A(s)}+\frac{2\mu u^0}{\rho_b} \hat{p}+\overline{H}(U, \psi, D\psi)&\text{in}\quad \mathcal{M},\\[3pt]
\hat{E}=E^--E_b^-&\text{on}\quad \mathcal{M}_0;
\end{cases}\\
&&\begin{cases}\label{eq378}
D_uA(s)=0&\text{in}\quad\mathcal{M},\\[3pt]
i^*(\widehat{A(s)})=\frac{\mu_4}{\mu_2}i^*(\hat{p})+\bar{g}_4
-\frac{\mu_4}{\mu_2}\bar{g}_2&\text{on}\quad \mathcal{M}_0.
\end{cases}
\end{eqnarray}
Here we have replaced the boundary condition in \eqref{eq367} by \eqref{eq344s}, and $ \bar{g}_2=g_2$.

For the vector field $u$ defined in $\mathcal{M}$, we consider the non-autonomous vector field $\frac{u'}{u^0}(y^0, y')$ defined for $y'=(y^1, y^2)\in \T^2$ and $y^0\in[r_b, L]$. For $\bar{y}\in \T^2$, we write the integral curve passing $(r_b, \bar{y})$ as $y'=\varphi(y^0, \bar{y})$, which is a $C^{k,\alpha}$ function in $\mathcal{M}$ if $u\in C^{k,\alpha}(\overline{\mathcal{M}})$ and $u^0>\delta$ for a positive  constant $\delta$, and $k\in\mathbb{N}$. For fixed $y^0$, the map $\varphi_{y^0}: \T^2\to \T^2, \ \bar{y}\mapsto y'=\varphi(y^0, \bar{y})$  is a $C^{k,\alpha}$ homeomorphism, then $\bar{y}=(\varphi_{y^0})^{-1}(y')$. Note that $\varphi_{r_b}$ is the identity map on $\T^2$. Also recall the following lemma appeared in \cite[p.733]{LiuXuYuan2016}:

\begin{lemma}\label{lem301}
Suppose that $u=(u^0,u')\in C^{0,1}$ and $u^0>\delta$. There is a positive constant $C=C(\delta,L-r_b)$ so that for any $y'\in \T^2$ and $y^0\in[r_b, L]$, it holds
\begin{eqnarray}\label{eq379}
\left|(\varphi_{y^0})^{-1}y'-y'\right|\le C\norm{u'}_{C^0(\overline{\mathcal{M}})}.
\end{eqnarray}
\end{lemma}

\begin{proof} There holds
$\left|\varphi_{y^0}(\bar{y})-\bar{y}\right|\le\int_{r_b}^{y^0}\left|\frac{u'}{u^0}(s, \varphi(s,\bar{y}))\right|\,\dd s\le C\norm{u'}_{C^0(\overline{\mathcal{M}})}$
as desired.
\end{proof}

We write the unique solutions to the linear transport equations \eqref{eq378E} \eqref{eq378} respectively as:
\begin{align}\label{eq380E}
\hat{E}(y)&=\hat{E}(y^0, y')=e^{2\mu (r_b-y^0)}(E^--E_b^-)(\bar{y})\nonumber\\
&\quad+\int_{r_b}^{y^0}e^{2\mu(\tau-y^0)}    \left(\frac{2\mu}{\gamma-1}\rho_{b}^{\gamma-1}\widehat{A(s)}+\frac{2\mu}{\rho_b} \hat{p}+\frac{1}{u^0}\overline{H}\right)(\tau, \varphi_\tau(\bar{y}))\,\dd \tau,\\
A(s)(y)&=A(s)(y^0, y')=(i^*A(s))(\bar{y}).\label{eq380}
\end{align}
Since the entropy is a constant behind the shock-front for the background solution, we have
\begin{align}\label{eq381}
\widehat{A(s)}(y)
&=i^*(\widehat{A(s)})(y')+\Big((i^*A(s))(\bar{y})-(i^*A(s))( y')\Big)\nonumber\\
&=\frac{\mu_4}{\mu_2}i^*(\hat{p})+\bar{g}_4-\frac{\mu_4}{\mu_2}\bar{g}_2+
\Big((i^*A(s))(\bar{y})-(i^*A(s))(y')\Big).
\end{align}

Now set
\begin{align}
{F}_7&=-\mu^2\rho_bd_3(t)\int_{r_b}^{y^0}e^{2\mu(\tau-y^0)}    \left(\frac{2\mu}{\rho_b} \Big(\hat{p}(\tau, \varphi_\tau(\bar{y}))-\hat{p}(\tau, y')\Big)+\frac{1}{u^0}\overline{H}\right)\,\dd \tau,\label{eq382E}\\
{F}_8&=-\mu^2\rho_bd_3(t)\int_{r_b}^{y^0}\frac{2\mu}{\gamma-1}e^{2\mu(\tau-y^0)}
\rho_{b}^{\gamma-1}(\tau, \varphi_\tau(\bar{y})) \,\dd \tau\nonumber\\
&\qquad\times\left\{\bar{g}_4-\frac{\mu_4}{\mu_2}\bar{g}_2+\Big((i^*A(s))(\bar{y})
-(i^*A(s))(y')\Big)\right\},
\label{eq382S}\\
{F}_9&=-\mu^2\rho_b^\gamma d_4(t)
\left\{\bar{g}_4-\frac{\mu_4}{\mu_2}\bar{g}_2+\Big((i^*A(s))(\bar{y})-(i^*A(s))(y')\Big)\right\},
\label{eq382}
\end{align}
which are higher-order terms (note that $\p_\beta \hat{p}$ and $\p_\beta A(s)$ are small, and $\varphi_{y^0}$ is close to the identity map since $u'$ is nearly zero, so $|(\varphi_{y^0})^{-1}(y')-y'|$ is small by \eqref{eq379}). Then
we could write the elliptic equation  \eqref{eq370} as
\begin{align}
&\Big(t(y^0)-1\Big)\p_0^2\hat{p}-\p_1^2\hat{p}-\p_2^2\hat{p}+\mu d_1(t(y^0))\p_0\hat{p}+\mu^2d_2(t(y^0))\hat{p}\nonumber\\
&+\, 2\mu^3e^{-2\mu y^0}\rho_b(y^0)d_3(t(y^0))
\int_{r_b}^{y^0}\frac{ e^{2\mu\tau}}{\rho_b(\tau)}\hat{p}(\tau, {y}')\,\dd \tau
\nonumber\\
&+\, \mu^2\frac{\mu_4}{\mu_2}\left(\rho_b(y^0)^\gamma d_4(t(y^0))+ \rho_b(y^0)d_3(t(y^0))\int_{r_b}^{y^0}\frac{2\mu}{\gamma-1}e^{2\mu(\tau-y^0)}\rho_{b}^{\gamma-1}(\tau)    \,\dd \tau\right)i^*(\hat{p})\nonumber\\
=&-\mu^2 e^{2\mu (r_b-y^0)}\rho_b(y^0)d_3(t(y^0))(E^--E_b^-)(\bar{y})
+\bar{F}_5+{F}_7+{F}_8+{F}_9.\label{eq383}
\end{align}
If we define
\begin{align*}
e_1(y^0)& = (t(y^0)-1)<0,\qquad e_2(y^0)={\mu}d_1(t(y^0)),\\
e_3(y^0)& = \mu^2d_2(t(y^0)),\qquad\quad
e_4(y^0)=2\mu^3e^{-2\mu y^0}\rho_b(y^0)d_3(t(y^0)),\\
e_5(y^0)& = \mu^2\frac{\mu_4}{\mu_2}\left(\rho_b(y^0)^\gamma d_4(t(y^0))+ \rho_b(y^0)d_3(t(y^0))\int_{r_b}^{y^0}\frac{2\mu}{\gamma-1}e^{2\mu(\tau-y^0)}
\rho_{b}^{\gamma-1}(\tau)    \,\dd \tau\right),\\
e_6(y^0)&=-\mu^2e^{2\mu (r_b-y^0)}\rho_b(y^0)d_3(t(y^0)),\, b(\mu,\tau)=\frac{ e^{2\mu \tau}}{\rho_b(\tau)}>0,\,  \hat{E}^{-}=(E^--E_b^-)(\bar{y}),\\
F& = F(U,\psi, DU,D^2p, D\psi, D^2\psi)\triangleq \overline{F}_5+{F}_7+{F}_8+{F}_9,
\end{align*}
then equation \eqref{eq383} simply reads
\begin{align}\label{eq388}
\mathfrak{L}(\hat{p})&\triangleq e_1(y^0)\p_0^2\hat{p}-\p_1^2\hat{p}-\p_2^2\hat{p}+e_2(y^0)\p_0\hat{p}+e_3(y^0)\hat{p}\nonumber \\
&\qquad+\, e_4(y^0)\int_{r_b}^{y^0}b(\mu,\tau) \hat{p}(\tau, {y}')\,\dd \tau+e_5(y^0)i^*(\hat{p})\nonumber \\ &=e_6(y^0)\hat{E}^-+F.
\end{align}
Note that there are nonlocal terms $e_4(y^0)\int_{r_b}^{y^0}{b}(\mu,\tau)\hat{p}(\tau, {y}')\,\dd \tau$ and $e_5(y^0)i^*(\hat{p})$. This is quite different from \cite{LiuXuYuan2016} and shows the spectacular influence of friction: which introduces stronger coupling in the Euler equations, resulted in stronger integral-type nonlocal terms. In conclusion, problem \eqref{eq365} can be reformulated as:
\begin{eqnarray}
\begin{cases}\label{eq389}
\mathfrak{L}(\hat{p})=e_6(y^0)\hat{E}^-+F& \text{in}\quad \mathcal{M},\\[3pt]
\hat{p}=p_1-p_b^+ &\text{on}\quad \mathcal{M}_1,\\[3pt]
\Delta'(i^*\hat{p})+\mu_7(i^*{\hat{p}})+\mu_8(i^*\p_0\hat{p})=\bar{g}_8  &\text{on}\quad \mathcal{M}_0.
\end{cases}
\end{eqnarray}

We then state Problem (T4), which is equivalent to Problem (T3), as can be seen from the above derivations.

\medskip
\fbox{
\parbox{0.90\textwidth}{
Problem (T4): Find $\psi\in C^{4,\alpha}(\T^2)$ and $U=U_b^++\hat{U}$ defined in $\mathcal{M}$ that solve problems \eqref{eq364}, \eqref{eq389} and \eqref{eq366}--\eqref{eq369}.}}

\section{A linear second-order nonlocal elliptic equation with Venttsel boundary condition}\label{sec4}
To attack problem \eqref{eq389},  we study in this section the following  linear second-order nonlocal elliptic equation subjected to a Venttsel boundary condition on $\mathcal{M}_0$:
\begin{eqnarray}\label{eq401}
\begin{cases}
\mathfrak{L}(\hat{p})=f(y)& \text{in}\quad \mathcal{M},\\[3pt]
\hat{p}=h_1(y') &\text{on}\quad \mathcal{M}_1,\\[3pt]
\Delta'(i^*\hat{p})+\mu_7(i^*{\hat{p}})+\mu_8(i^*\p_0\hat{p})=h_0(y') &\text{on}\quad \mathcal{M}_0.
\end{cases}
\end{eqnarray}
Here $f\in C^{k-2,\alpha}(\overline{\mathcal{M}}), h_1\in C^{k,\alpha}(\T^2)$ and $h_0\in C^{k-2,\alpha}(\T^2)$ are given nonhomogeneous terms, and $k=2,3,\ldots$

In 1959, A. D. Venttsel proposed the now called Venttsel problem of second-order elliptic equations \cite{Venttsel1959}, from the view point of probability theory. A survey of the mathematical studies of Venttsel problem could be found in \cite{AN-2000}. It is quite interesting to see that such Venttsel problems appear so naturally in the studies of transonic shocks, not only for the geometric effects considered in \cite{LiuXuYuan2016}, but also the effects of frictions considered in this paper.  As we mentioned in \cite{LiuXuYuan2016}, the linear theory established by  Luo and Trudinger in \cite{LN1991} cannot be applied directly to our problem, since the elliptic operator $\mathfrak{L}$ contains  nonclassical nonlocal terms, and the coefficients of the zeroth-order term may change sign. So in the following we mainly follow the procedure in \cite{LiuXuYuan2016} to study problem \eqref{eq401}.

\subsection{Uniqueness of solutions in Sobolev spaces}
We firstly study under what conditions a strong solution $\hat{p}$ in Sobolev space $H^2(\mathcal{M})$ with $i^*\hat{p}\in H^2(\mathcal{M}_0)$ to problem \eqref{eq401} is unique. To this end, we consider the homogeneous problem
\begin{eqnarray}\label{eq402}
\begin{cases}
\mathfrak{L}(\hat{p})=0& \text{in}\quad \mathcal{M},\\[3pt]
\hat{p}=0 &\text{on}\quad \mathcal{M}_1,\\[3pt]
\Delta'(i^*\hat{p})+\mu_7(i^*{\hat{p}})+\mu_8(i^*\p_0\hat{p})=0 &\text{on}\quad \mathcal{M}_0.
\end{cases}
\end{eqnarray}

\begin{lemma}[Regularity]\label{lem4001}
If $\hat{p}\in H^2(\mathcal{M})$ with $i^*\hat{p}\in H^2(\mathcal{M}_0)$ is a strong solution to \eqref{eq402}, then $\hat{p}$ is a classical solution and  belongs to $C^\infty(\overline{\mathcal{M}})$.
\end{lemma}

\begin{proof} The argument is the same as that appeared in \cite[Remark 5.5, p.735]{LiuXuYuan2016}, except that we consider now $e_3(y^0)\hat{p}+e_4(y^0)\int_{r_b}^{y^0}b(\mu,\tau) \hat{p}(\tau, {y}')\,\dd \tau+e_5(y^0)i^*(\hat{p})$  as  a nonhomogeneous term and apply the Schauder theory in \cite{GT2001}, by noting that the term $e_4(y^0)\int_{r_b}^{y^0}b(\mu,\tau) \hat{p}(\tau, {y}')\,\dd \tau$ still belongs to $C^\alpha(\overline{\mathcal{M}}\setminus\mathcal{M}_0)$ for any $\alpha\in(0,1/2)$, thanks to Sobolev embedding theorem.
\end{proof}

We wish to find some conditions to guarantee that $\hat{p}\equiv0$. The idea is to use the method of separation of variables via the Fourier series. Similar to \cite{YuanZhao},
denote $m=(m_1,m_2)$, then by the above lemma, we could  write
\begin{align}\label{eq404}
\hat{p}(y)=&\sum_{m_1,m_2=0}^{\infty}\lambda_{m}\left\{ p_{1,m}(y^0)\cos(m_1y^1)\cos(m_2y^2)+p_{2,m}(y^0)\sin(m_1y^1)\cos(m_2y^2)\right.\nonumber\\
&\qquad\left.+p_{3,m}(y^0)\cos(m_1y^1)\sin(m_2y^2)+p_{4,m}(y^0)\sin(m_1y^1)\sin(m_2y^2)\right\},
\end{align}
where
\begin{eqnarray*}
\lambda_{m}=
\begin{cases}
\frac{1}{4}& \text{if}\quad m_1=m_2=0,\\[3pt]
\frac{1}{2} &\text{if only one of}\ \ m_1,\, m_2\ \ \text{is}\ \ 0,\\[3pt]
1 &\text{if}\quad m_1>0,m_2>0.
\end{cases}
\end{eqnarray*}
For $y=(y^0, y^1,y^2)$, the coefficients in \eqref{eq404} are given by
\begin{align*}
p_{1,m}(y^0)=&\frac{1}{\pi^2}\int_{\T^2}\hat{p}(y)\cos(m_1y^1)\cos(m_2y^2)\,\dd y^1\dd y^2,\\
p_{2,m}(y^0)=&\frac{1}{\pi^2}\int_{\T^2}\hat{p}(y)\sin(m_1y^1)\cos(m_2y^2)\,\dd y^1\dd y^2,\\
p_{3,m}(y^0)=&\frac{1}{\pi^2}\int_{\T^2}\hat{p}(y)\cos(m_1y^1)\sin(m_2y^2)\,\dd y^1\dd y^2,\\
p_{4,m}(y^0)=&\frac{1}{\pi^2}\int_{\T^2}\hat{p}(y)\sin(m_1y^1)\sin(m_2y^2)\,\dd y^1\dd y^2.
\end{align*}
If $\hat{p}\in C^{k,\alpha}(\overline{\mathcal{M}})$ and $k\ge 2$, we easily deduce that  $p_{i,m}(y^0)$ $(i=1,2,3,4)$ belongs to $C^{k,\alpha}([r_b, L])$ for all $k$, and the Fourier series \eqref{eq404} converges uniformly in $C^{k-2}(\overline{\mathcal{M}})$.

Substituting \eqref{eq404} into \eqref{eq402},   for $y^0\in[r_b, L]$,  each $p_{i,m}(y^0)$ solves the following nonlocal ordinary differential equation
\begin{eqnarray}\label{eq405}
&&e_1(y^0)p_{i,m}''+e_2(y^0)p_{i,m}'+(e_3(y^0)
+|m|^2)p_{i,m}\nonumber\\
&&\qquad+e_4(y^0)\int_{r_b}^{y^0}{b}(\mu,\tau) p_{i,m}(\tau)\,\dd \tau+e_5(y^0)p_{i,m}(r_b)=0,
\end{eqnarray}
subjected to the two-point boundary conditions:
\begin{eqnarray}\label{eq406}
p_{i,m}'(r_b)+\frac{\mu_7-|m|^2}{\mu_8}p_{i,m}(r_b)=0 ,\qquad p_{i,m}(L)=0.
\end{eqnarray}
We need to find sufficient conditions so  that all $p_{i,m}\,  ( i=1,2,3,4)$ are zero.

Supposing that $p_{i,m}(r_b)=0,$  we set $\mathcal{P}_{i,m}(y^0)=\int_{r_b}^{y^0}{b}(\mu,\tau) p_{i,m}(\tau)\,\dd \tau$. Then problem \eqref{eq405} and \eqref{eq406} can be written as
\begin{eqnarray}\label{eq407}
\begin{cases}
\tilde{e}_1(y^0)\mathcal{P}_{i,m}'''+\tilde{e}_2(y^0)\mathcal{P}''_{i,m}+\tilde{e}_3(y^0)\mathcal{P}'_{i,m}
+e_4(y^0)\mathcal{P}_{i,m}=0,\quad y^0\in[r_b, L],\\[3pt]
\mathcal{P}_{i,m}(r_b)=\mathcal{P}_{i,m}'(r_b)=\mathcal{P}_{i,m}''(r_b)=0,\\[3pt]
\mathcal{P}_{i,m}'(L)=0.
\end{cases}
\end{eqnarray}
Here we define
\begin{align*}
\tilde{e}_1(y^0)=&\frac{e_1(y^0)}{{b}(\mu,y^0)}<0,\\
\tilde{e}_2(y^0)=&\left(\frac{e_2(y^0)}{{b}(\mu,y^0)}-\frac{2e_1(y^0){b}'(\mu,y^0)}{{b}^2(\mu,y^0)}
\right),\\
\tilde{e}_3(y^0)=&\left(\frac{(e_3(y^0)+|m|^2)}{{b}(\mu,y^0)}-\frac{e_2(y^0){b}'(\mu,y^0)
+e_1(y^0){b}''(\mu,y^0)}{{b}^2(\mu,y^0)}
+\frac{2e_1(y^0)({b}'(\mu,y^0))^2}{{b}^3(\mu,y^0)}
\right),
\end{align*}
and $b'(\mu,y_0)=\p b(\mu,y_0)/\p y^0, b''(\mu,y_0)=\p^2 b(\mu,y_0)/{(\p y^0)}^2$.
By uniqueness of solutions of Cauchy problems of ordinary differential equations, obviously one has that $p_{i,m}\equiv 0$ in $[r_b, L]$.

If $p_{i,m}(r_b)\ne0$ for some $i, m$, we set $w_{i,m}(y^0)=\frac{p_{i,m}(y^0)}{p_{i,m}(r_b)}$ and $\mathcal{W}_{i,m}(y^0)=\int_{r_b}^{y^0}{b}(\mu,\tau) w_{i,m}(\tau)\,\dd \tau$. Then it solves
\begin{eqnarray}\label{eq408}
\begin{cases}
\tilde{e}_1(y^0)\mathcal{W}_{i,m}'''+\tilde{e}_2(y^0)\mathcal{W}''_{i,m}
+\tilde{e}_3(y^0)\mathcal{W}'_{i,m}+e_4(y^0)\mathcal{W}_{i,m}+e_5(y^0)=0,\quad y^0\in[r_b, L],\\[3pt]
\mathcal{W}_{i,m}(r_b)=0,\quad \mathcal{W}_{i,m}'(r_b)=b(\mu,r_b),
\quad \mathcal{W}_{i,m}''(r_b)=b'(\mu,r_b)-\frac{\mu_7-|m|^2}{\mu_8}b(\mu,r_b),\\[3pt]
\mathcal{W}_{i,m}'(L)=0.
\end{cases}
\end{eqnarray}

\begin{definition}\label{def401}
We say a background solution $U_b$ satisfies the {\it S-Condition}, if for each $i=1,2,3,4$ and $m\in\Z^2$,  problem \eqref{eq408} does {\it not} have a classical solution.
\end{definition}

If the background solution $U_b$ satisfies the S-Condition, then all $p_{i,m}$ are zero, hence problem \eqref{eq402} has only the trivial solution. Recall that a background  solution is determined analytically by the following parameters: $L>0, \gamma>1, \mu>0, r_b\in(0, L), U_b^-(0)$. Our purpose below is to show theoretically that almost all background solutions satisfy the S-Condition. Actually, we have the following lemma.

\begin{lemma}\label{lem401}
Given $U_b^-(0)$ and $\gamma>1, \mu>0, L>0$. Suppose that the Mach number of the background solution satisfies $t(y^0)<t_0, y^0\in(r_b, L)$, for a constant $t_0$ depending on $\gamma$. Then there exists a set $\mathcal{S}\subset(0, L)$ of at most countable infinite points such that the background solutions $U_b$ determined by $r_b\in(0, L)\setminus \mathcal{S}$ satisfy the S-Condition.
\end{lemma}

\begin{proof}
The idea of proof is similar to the previous work  \cite{ChenYuan2013,LiuXuYuan2016}. Note that there exists a constant $t_0$  depending only on $\gamma$ such that if the Mach number of the background solution $t(y^0)\ne t_0, y^0\in(r_b, L)$, then $e_4(y^0)\neq 0$.
Let $\widetilde{\mathcal{W}}_{i,m}=\mathcal{W}_{i,m}'(y^0)$. Firstly dividing by $e_4(y_0)$ in \eqref{eq408},  and then taking derivative of the resulting equation,  we get, after multiplying $e_4(y^0)$, that
\begin{eqnarray}\label{eq408-1}
\begin{cases}
\tilde{e}_1(y^0)\widetilde{\mathcal{W}}_{i,m}'''+\big(\tilde{e}_2(y^0)+\tilde{\tilde{e}}_1(y^0)\big)
\widetilde{\mathcal{W}}''_{i,m}
+\big(\tilde{e}_3(y^0)+\tilde{\tilde{e}}_2(y^0)\big)\widetilde{\mathcal{W}}'_{i,m}\\[3pt]
\hspace{14.7em}+\big(e_4(y^0)+\tilde{\tilde{e}}_3(y^0)\big)\widetilde{\mathcal{W}}_{i,m}
+\widetilde{e}_5(y^0)=0,\quad y^0\in[r_b, L],\\[3pt]
\widetilde{\mathcal{W}}_{i,m}(r_b)=b(\mu,r_b),\quad \widetilde{\mathcal{W}}_{i,m}'(r_b)=b'(\mu,r_b)-\frac{\mu_7-|m|^2}{\mu_8}b(\mu,r_b),\\[3pt]
\widetilde{\mathcal{W}}_{i,m}''(r_b)=-\frac{1}{\tilde{e}_1(r_b)}\Big(\tilde{e}_2(r_b)
\big(b'(\mu,r_b)-\frac{\mu_7-|m|^2}{\mu_8}b(\mu,r_b)\big)
+\tilde{e}_3(r_b)b(\mu,r_b)+e_5(r_b)\Big),\\[3pt]
\widetilde{\mathcal{W}}_{i,m}(L)=0.
\end{cases}
\end{eqnarray}
where
\begin{align*}
\tilde{\tilde{e}}_1(y^0)=&e_4(y^0)\left(\frac{\tilde{e}_1(y^0)}{e_4(y^0)}\right)',\quad
\tilde{\tilde{e}}_2(y^0)=e_4(y^0)\left(\frac{\tilde{e}_2(y^0)}{e_4(y^0)}\right)',\\
\tilde{\tilde{e}}_3(y^0)=&e_4(y^0)\left(\frac{\tilde{e}_3(y^0)}{e_4(y^0)}\right)',\quad
\tilde{e}_5(y^0)=e_4(y^0)\left(\frac{e_5(y^0)}{e_4(y^0)}\right)'.
\end{align*}
Now we change the variable $y^0$ to $z$ given by $z=\frac{y^0-r_b}{L-r_b}$, and $\tilde{z}=r_b+(L-r_b)z$. Then by multiplying suitable powers of $L-r_b$, the above problem \eqref{eq408-1} becomes
\begin{eqnarray}\label{eq408-2}
\begin{cases}
\tilde{e}_1(\tilde{z})\widehat{\mathcal{W}}_{i,m}'''+(L-r_b)
\big(\tilde{e}_2(\tilde{z})+\tilde{\tilde{e}}_1(\tilde{z})\big)\widehat{\mathcal{W}}''_{i,m}
+(L-r_b)^2\big(\tilde{e}_3(\tilde{z})+\tilde{\tilde{e}}_2(\tilde{z})\big)\widehat{\mathcal{W}}'_{i,m}\\[3pt]
\hspace{10.7em}+(L-r_b)^3\big(e_4(\tilde{z})+\tilde{\tilde{e}}_3(\tilde{z})\big)\widehat{\mathcal{W}}_{i,m}
+(L-r_b)^3\widetilde{e}_5(\tilde{z})=0,\quad z\in[0,1],\\[3pt]
\widehat{\mathcal{W}}_{i,m}(0)=b(\mu,r_b),\quad \widehat{\mathcal{W}}_{i,m}'(0)=(L-r_b)\Big(b'(\mu,r_b)-\frac{\mu_7-|m|^2}{\mu_8}b(\mu,r_b)\Big),\\[3pt]
\widehat{\mathcal{W}}_{i,m}''(0)=-\frac{(L-r_b)^2}{\tilde{e}_1(r_b)}\Big(\tilde{e}_2(r_b)
\big(b'(\mu,r_b)-\frac{\mu_7-|m|^2}{\mu_8}b(\mu,r_b)\big)
+\tilde{e}_3(r_b)b(\mu,r_b)+e_5(r_b)\Big),\\[3pt]
\widehat{\mathcal{W}}_{i,m}(1)=0.
\end{cases}
\end{eqnarray}
Here we have set $\widehat{\mathcal{W}}_{i,m}(z)=\widetilde{\mathcal{W}}_{i,m}(r_b+z(L-r_b))$.

We recall that the background solution $U_b$, and all the coefficients $e_1, e_2, e_3, e_4$, as well as $b$,  depend analytically on $r_b$. Hence the unique solution $\widehat{\mathcal{W}}_{i,m}$ to this Cauchy problem \eqref{eq408-2} (exclude the last condition) is also real analytic with respect to the parameter $r_b$ (cf. \cite{Walter1998}). We write it as $\widehat{\mathcal{W}}_{i,m}=\widehat{\mathcal{W}}_{i,m}(z; r_b)$. Particularly, $\vartheta_{i,m}(r_b)\triangleq \widehat{\mathcal{W}}_{i,m}(1; r_b)$ is continuous for $r_b\in[0, L],$ and real analytic for $r_b\in[0, L)$.

It is crucial to note that there is an  analytical continuation of   $\vartheta_{i,m}(r_b)$ up to the point $r_b=L$.  This follows from the observation in Remark \ref{rem11new}, by which all the coefficients, nonhomogeneous terms, and boundary conditions in \eqref{eq408-2} make sense for $r_b>L$ (but close to $L$), and are actually analytic for $r_b$ in a neighborhood of $L$.

For given $i=1,2,3,4, m\in\Z^2$, suppose now there are infinite numbers of  $r_b$ so that $\vartheta_{i,m}(r_b)=0$. Then by compactness of $[0, L]$, the function $\vartheta_{i,m}$, as an analytic function, has a non-isolated zero point. So it must be identically zero and we have $\vartheta_{i,m}(L)=0.$ However, for $r_b=L$, problem \eqref{eq408-2} is reduced to
\begin{eqnarray}\label{eq409}
\begin{cases}
\widehat{\mathcal{W}}_{i,m}'''(z)=0,\quad z\in[0,1],\\[3pt]
\widehat{\mathcal{W}}_{i,m}(0)=b(\mu,L),\quad \widehat{\mathcal{W}}_{i,m}'(0)=0,\quad
\widehat{\mathcal{W}}_{i,m}''(0)=0.
\end{cases}
\end{eqnarray}
Hence $\widehat{\mathcal{W}}_{i,m}(1; L)=b(\mu,L)$, namely $\vartheta_{i,m}(L)=b(\mu,L)=e^{2\mu L}/\rho_b(L)>0$, contradicts to our conclusion that $\vartheta_{i,m}(L)=0$. So for each fixed $m\in\Z^2$, there are at most finite numbers of zeros of $\vartheta_{i,m}$. Therefore, there are at most countable infinite numbers of $r_b$ so that the problem \eqref{eq408} may have a solution. The conclusion of the lemma then follows.
\end{proof}

\subsection{Uniform a priori estimate in Sobolev spaces}
Suppose now that $\hat{p}\in H^2(\mathcal{M})$ with $i^*(\hat{p})\in H^2(\mathcal{M}_0)$. Obviously our assumptions on problem \eqref{eq401} guarantee that $f\in L^2(\mathcal{M})$, $h_1\in H^2(\T^2)$, and $h_0\in L^2(\T^2)$.  Then by Trace Theorem and Interpolation Inequalities of Sobolev functions, we have
\begin{eqnarray*}
\norm{i^*(\hat{p})}_{L^2(\T^2)}+\norm{i^*(\p_0\hat{p})}_{L^2(\T^2)}
\le\varepsilon\norm{\hat{p}}_{H^2(\mathcal{M})}+C(\varepsilon)\norm{\hat{p}}_{L^2(\mathcal{M})},
\quad \forall \ \varepsilon\in(0,1).
\end{eqnarray*}
Applying Theorem 8.12 in \cite[p.186]{GT2001}  to the boundary equation in \eqref{eq401}, we have
\begin{align*}
\norm{i^*(\hat{p})}_{H^2(\T^2)}\le& C\Big(\norm{i^*(\hat{p})}_{L^2(\T^2)}+\norm{i^*(\p_0\hat{p})}_{L^2(\T^2)}+\norm{h_0}_{L^2(\T^2)}\Big)
\nonumber\\
\le&C\varepsilon\norm{\hat{p}}_{H^2(\mathcal{M})}
+C'(\varepsilon)\norm{\hat{p}}_{L^2(\mathcal{M})}+C\norm{h_0}_{L^2(\T^2)}.
\end{align*}
By considering the nonlocal terms
\[
e_4(y^0)\int_{r_b}^{y^0}b(\mu,\tau) \hat{p}(\tau, {y}')\,\dd \tau\quad \text{and} \quad e_5(y^0)i^*(\hat{p})
\]
in $\mathfrak{L}(\hat{p})$  as part of the non-homogenous term,  and
using the same theorem to problem \eqref{eq401}, with given Dirichlet data $i^*\hat{p}$,  it follows that
\begin{align*}
\norm{\hat{p}}_{H^2(\mathcal{M})}&\le C\Big(\norm{\hat{p}}_{L^2(\mathcal{M})}+\norm{i^*(\hat{p})}_{L^2(\T^2)}
+\norm{f}_{L^2(\mathcal{M})}+\norm{h_1}_{H^2(\T^2)}\Big)\nonumber\\
&\le C\varepsilon\norm{\hat{p}}_{H^2(\mathcal{M})}+C'(\varepsilon)\norm{\hat{p}}_{L^2(\mathcal{M})}+
C\Big(\norm{h_1}_{H^2(\T^2)}+\norm{f}_{L^2(\mathcal{M})}\Big).
\end{align*}
Taking $\varepsilon=1/(4C)$, we get
\begin{eqnarray}\label{eq415}
\norm{\hat{p}}_{H^2(\mathcal{M})}+\norm{i^*(\hat{p})}_{H^2(\T^2)}\le
C\Big(\norm{\hat{p}}_{L^2(\mathcal{M})}+\norm{h_0}_{L^2(\T^2)}+\norm{h_1}_{H^2(\T^2)}
+\norm{f}_{L^2(\mathcal{M})}\Big).
\end{eqnarray}
Then, by \eqref{eq415} and a compactness argument as in \cite[p.738]{LiuXuYuan2016},  we deduce the {\it a priori} estimate
\begin{eqnarray}\label{eq416}
\norm{\hat{p}}_{H^2(\mathcal{M})}+\norm{i^*(\hat{p})}_{H^2(\T^2)}\le
C\Big(\norm{h_0}_{L^2(\T^2)}+\norm{h_1}_{H^2(\T^2)}+\norm{f}_{L^2(\mathcal{M})}\Big),
\end{eqnarray}
provided that the S-Condition holds. Here the constant $C$ depends only on the background solution.

\subsection{Uniform a priori estimate in H\"{o}lder spaces}
By considering the nonlocal terms
\[
e_4(y^0)\int_{r_b}^{y^0}b(\mu,\tau) \hat{p}(\tau, {y}')\,\dd \tau \ \ \text{and}\ \  e_5(y^0)i^*(\hat{p})
\]
in $\mathfrak{L}(\hat{p})$  as part of the non-homogenous term,  and
applying Theorem 1.5 in \cite[p.198]{LN1991} for the Venttsel problem (note that $\mu_8>0$) and Theorem 6.6 in \cite{GT2001} for the Dirichlet problem, with the aid of a standard higher regularity
argument as in Theorem 6.19 of \cite{GT2001}, and interpolation inequalities (Lemma 6.35 in \cite[p.135]{GT2001}),  we infer that  any $\hat{p}\in C^{k,\alpha}(\overline{\mathcal{M}})$ ($k=2,3$) solves problem \eqref{eq401} should satisfy the estimate
\begin{eqnarray}\label{eq410}
\norm{\hat{p}}_{C^{k,\alpha}(\overline{\mathcal{M}})}\le
C\Big(\norm{\hat{p}}_{C^0(\overline{\mathcal{M}})}+\norm{h_0}_{C^{k-2,\alpha}(\T^2)}
+\norm{h_1}_{C^{k,\alpha}(\T^2)}+\norm{f}_{C^{k-2,\alpha}(\overline{\mathcal{M}})}
\Big),
\end{eqnarray}
with  $C$ a constant depending only on the background solution $U_b$ and $L, \alpha$.
Then by an argument similar to the proof of \eqref{eq416}, we have the {\it a priori} estimate:
\begin{eqnarray}\label{eq411}
\norm{\hat{p}}_{C^{k,\alpha}(\overline{\mathcal{M}})}\le
C\Big(\norm{h_0}_{C^{k-2,\alpha}(\T^2)}+\norm{h_1}_{C^{k,\alpha}(\T^2)}
+\norm{f}_{C^{k-2,\alpha}(\overline{\mathcal{M}})}\Big)
\end{eqnarray}
for any $C^{k,\alpha}$ solution of problem \eqref{eq401}, provided that the only solution to problem \eqref{eq402} is zero.

\subsection{Approximate solutions}
We now use Fourier series to establish a family of approximate solutions to problem \eqref{eq401}.

Without loss of generality, we take $h_1=0$  in the sequel. We also set $\{f^{(n)}\}_n$ to be a sequence of $C^\infty(\overline{\mathcal{M}})$ functions that converges to $f$ in $C^{k-2,\alpha}(\overline{\mathcal{M}})$, and $\{h^{(n)}_0\}_n\subset C^\infty(\T^2)$ converges to $h_0$ in $C^{k-2,\alpha}(\T^2)$. Now for fixed $n$, we consider problem \eqref{eq401}, with $f$ there replaced by $f^{(n)}$, and $h_0$ replaced by $h^{(n)}_0$.

Suppose that
\begin{align}
f^{(n)}(y)=&\sum_{m_1,m_2=0}^{\infty}\lambda_{m}\left\{ f^{(n)}_{1,m}(y^0)\cos(m_1y^1)\cos(m_2y^2)+f^{(n)}_{2,m}(y^0)\sin(m_1y^1)
\cos(m_2y^2)\right.\nonumber\\
&\left.+f^{(n)}_{3,m}(y^0)\cos(m_1y^1)\sin(m_2y^2)+f^{(n)}_{4,m}(y^0)\sin(m_1y^1)
\sin(m_2y^2)\right\},\label{eq417}\\
h^{(n)}_0(y')=&\sum_{m_1,m_2=0}^{\infty}\lambda_{m}\left\{ (h_0^{(n)})_{1,m}\cos(m_1y^1)\cos(m_2y^2)+(h_0^{(n)})_{2,m}\sin(m_1y^1)
\cos(m_2y^2)\right.\nonumber\\
&\left.+(h_0^{(n)})_{3,m}\cos(m_1y^1)\sin(m_2y^2)+(h_0^{(n)})_{4,m}\sin(m_1y^1)
\sin(m_2y^2)\right\}.\label{eq418}
\end{align}
Then for $\hat{p}$ given by \eqref{eq404}, each $\mathcal{P}_{i,m}(y^0), i=1,2,3,4$  should solve the following two-point boundary value problem of a third-order ordinary differential equation containing a nonlocal term:
\begin{eqnarray}\label{eq419}
\begin{cases}
L_{i,m}(\mathcal{P}_{i,m})\triangleq\tilde{e}_1(y^0)\mathcal{P}_{i,m}'''+\tilde{e}_2(y^0)\mathcal{P}''_{i,m}
+\tilde{e}_3(y^0)\mathcal{P}'_{i,m}+e_4(y^0)\mathcal{P}_{i,m}\\[3pt]
\hspace{4.7em}=-\frac{e_5(y^0)}{b(\mu,r_b)}\mathcal{P}'_{i,m}(r_b)+{f}^{(n)}_{i,m}(y^0)\qquad y^0\in[r_b, L],\\[3pt]
\mathcal{P}_{i,m}(r_b)=0,
\quad \mathcal{P}_{i,m}''(r_b)
+\left(\frac{\mu_7-|m|^2}{\mu_8}-\frac{b'(\mu,r_b)}{b(\mu,r_b)}\right)\mathcal{P}_{i,m}'(r_b)
=\frac{b(\mu,r_b)}{\mu_8}(h_0^{(n)})_{i,m},\\[3pt]
\mathcal{P}_{i,m}'(L)=0.
\end{cases}
\end{eqnarray}
We will show that this problem is uniquely solvable.

Let $\mathcal{P}^1_{i,m}, \mathcal{P}^2_{i,m}$ and $ \mathcal{P}^\flat_{i,m}$  be respectively the unique solutions of the following three linear Cauchy problems:
\begin{align*}
&L_{i,m}(\mathcal{P}_{i,m})=-\frac{e_5(y^0)}{b(\mu,r_b)},\quad \mathcal{P}_{i,m}(r_b)=0,\quad \mathcal{P}'_{i,m}(r_b)=1,\quad \mathcal{P}_{i,m}''(r_b)=0;\\
&L_{i,m}(\mathcal{P}_{i,m})=0,\quad \mathcal{P}_{i,m}(r_b)=0,\quad \mathcal{P}'_{i,m}(r_b)=0,\quad \mathcal{P}_{i,m}''(r_b)=1;\\
&L_{i,m}(\mathcal{P}_{i,m})={f}^{(n)}_{i,m}(y^0),\quad \mathcal{P}_{i,m}(r_b)=0,\quad \mathcal{P}'_{i,m}(r_b)=0,\quad \mathcal{P}_{i,m}''(r_b)=0.
\end{align*}
For any real numbers $c_1, c_2$,
$$\mathcal{P}_{i,m}=c_1\mathcal{P}^1_{i,m}+c_2\mathcal{P}^2_{i,m}+\mathcal{P}^\flat_{i,m}$$
solves the Cauchy problem
\begin{align*}
L_{i,m}(\mathcal{P}_{i,m})=-c_1\frac{e_5(y^0)}{b(\mu,r_b)}+{f}^{(n)}_{i,m}(y^0),
\quad \mathcal{P}_{i,m}(r_b)=0,\quad \mathcal{P}'_{i,m}(r_b)=c_1,\quad \mathcal{P}_{i,m}''(r_b)=c_2.
\end{align*}

Therefore, to solve problem \eqref{eq419}, there shall exist $c_1, c_2$ to solve the following linear algebraic equations:
\begin{align*}
c_2+\left(\frac{\mu_7-|m|^2}{\mu_8}-\frac{b'(\mu,r_b)}{b(\mu,r_b)}\right)c_1
&=\frac{b(\mu,r_b)}{\mu_8}(h_0^{(n)})_{i,m},\\
(\mathcal{P}^2_{i,m})'(L)c_2+(\mathcal{P}^1_{i,m})'(L)c_1&=-(\mathcal{P}^\flat_{i,m})'(L).
\end{align*}
In fact, we know that, under the S-Condition, the homogeneous system has only the trivial solution. So by Fredholm alternative  of linear algebraic equations, there is one and only one pair $(c_1, c_2)$ solves the above linear system, which enables us to get the unique solution to problem \eqref{eq419}.

Note that ${f}^{(n)}_{i,m}\in C^{\infty}([r_b, L])$ as  $f^{(n)}\in C^{\infty}(\overline{\mathcal{M}})$, and the coefficients in \eqref{eq419} are all real analytic, so the solution $p_{i,m}(y^0)=\frac{1}{b(\mu,y^0)}\mathcal{P}_{i,m}'(y^0)$ belongs to  $C^{\infty}([r_b, L]).$

Now for $N\in\mathbb{N}$, we define
\begin{align*}
\hat{p}_N(y)=&\sum_{m_1,m_2=0}^{N}\lambda_{m}\left\{ p_{1,m}(y^0)\cos(m_1y^1)\cos(m_2y^2)+p_{2,m}(y^0)\sin(m_1y^1)\cos(m_2y^2)\right.\\
&\left.+p_{3,m}(y^0)\cos(m_1y^1)\sin(m_2y^2)+p_{4,m}(y^0)\sin(m_1y^1)\sin(m_2y^2)\right\},\\
f^{(n)}_N(y)=&\sum_{m_1,m_2=0}^{N}\lambda_{m}\left\{ f^{(n)}_{1,m}(y^0)\cos(m_1y^1)\cos(m_2y^2)+f^{(n)}_{2,m}(y^0)\sin(m_1y^1)
\cos(m_2y^2)\right.\\
&\left.+f^{(n)}_{3,m}(y^0)\cos(m_1y^1)\sin(m_2y^2)+f^{(n)}_{4,m}(y^0)\sin(m_1y^1)
\sin(m_2y^2)\right\},\\
(h^{(n)}_0)_N(y')=&\sum_{m_1,m_2=0}^{N}\lambda_{m}\left\{ (h_0^{(n)})_{1,m}\cos(m_1y^1)\cos(m_2y^2)+(h_0^{(n)})_{2,m}\sin(m_1y^1)
\cos(m_2y^2)\right.\\
&\left.+(h_0^{(n)})_{3,m}\cos(m_1y^1)\sin(m_2y^2)+(h_0^{(n)})_{4,m}\sin(m_1y^1)
\sin(m_2y^2)\right\}.
\end{align*}
Apparently  $\hat{p}_{N}, \ f^{(n)}_N\in C^{\infty}(\overline{\mathcal{M}})$, and $(h_0^{(n)})_{N}\in C^\infty(\T^2)$. It is also easy to check that $\hat{p}_{N}$ solves the following problem:
\begin{eqnarray}\label{eq423}
\begin{cases}
\mathfrak{L}(\hat{p}_N)=f^{(n)}_N& \text{in}\quad \mathcal{M},\\[3pt]
\hat{p}_N=0 &\text{on}\quad \mathcal{M}_1,\\[3pt]
\Delta'(i^*\hat{p}_N)+\mu_7(i^*{\hat{p}_N})+\mu_8(i^*\p_0\hat{p}_N)=(h_0^{(n)})_N &\text{on}\quad \mathcal{M}_0.
\end{cases}
\end{eqnarray}

\subsection{Existence}\label{sec35}
By the estimate \eqref{eq416}, for any $N_1,N_2\in\mathbb{N}$ with $N_1<N_2$, there holds
\begin{align*}
&\norm{\hat{p}_{N_2}-\hat{p}_{N_1}}_{H^2(\mathcal{M})}
+\norm{i^*(\hat{p}_{N_2}-\hat{p}_{N_1})}_{H^2(\mathcal{M}_0)}\nonumber\\
\le& C\Big(\norm{f^{(n)}_{N_2}-f^{(n)}_{N_1}}_{L^2(\mathcal{M})}
+\norm{(h_0^{(n)})_{N_2}-(h_0^{(n)})_{N_1}}_{L^2(\T^2)}\Big).
\end{align*}
Recall that $f^{(n)}_N\to f^{(n)}$ in $L^2(\mathcal{M})$ and $(h_0^{(n)})_N\to h_0^{(n)}$ in $L^2(\T^2)$ as $N\to\infty$, we infer that  $\{\hat{p}_N\}$  (respectively $i^*{\hat{p}_N}$) is a Cauchy sequence in $H^2(\mathcal{M})$ (respectively $H^2(\mathcal{M}_0)$). So there is a $\hat{p}^{(n)}\in H^2(\mathcal{M})$ (respectively $q^{(n)}\in H^2(\mathcal{M}_0)$) and $\hat{p}_N\to \hat{p}^{(n)}$ in $H^2(\mathcal{M})$ (respectively $i^*\hat{p}_N\to q^{(n)}$ in $H^2(\mathcal{M}_0)$)  as $N\to \infty$. By continuity of trace operator, we conclude that $q^{(n)}=i^*\hat{p}^{(n)}$. Taking the limit $N\to\infty$ in problem  \eqref{eq423}, one sees that $\hat{p}^{(n)}\in H^2(\mathcal{M})$, with $i^*\hat{p}^{(n)}\in H^2(\mathcal{M}_0)$, is a strong solution to problem \eqref{eq401}, where $f$ is replaced by $f^{(n)}$, and $h_0$ replaced by $h_0^{(n)}$. Then by the same arguments as in Lemma \ref{lem401}, $\hat{p}^{(n)}\in C^{\infty}(\overline{\mathcal{M}})$ and of course it satisfies the estimate \eqref{eq411}.

Now for the approximate solutions $\{\hat{p}^{(n)}\}_n$, we use the estimate \eqref{eq411} to infer that
\begin{align*}
\norm{\hat{p}^{(n)}}_{C^{k,\alpha}(\overline{\mathcal{M}})}\le& C\Big(\norm{f^{(n)}}_{C^{k-2,\alpha}(\overline{\mathcal{M}})}
+\norm{h_0^{(n)}}_{C^{k-2,\alpha}(\mathcal{M}_0)}\Big)\nonumber\\
\le&C\Big(\norm{f}_{C^{k-2,\alpha}(\overline{\mathcal{M}})}+\norm{h_0}_{C^{k-2,\alpha}(\mathcal{M}_0)}\Big).
\end{align*}
Hence by Ascoli--Arzela Lemma, there is a subsequence of $\{\hat{p}^{(n)}\}$ that converges to some $\hat{p}\in C^{k,\alpha}(\overline{\mathcal{M}})$ in the norm of $C^{k}(\overline{\mathcal{M}})$. Taking limit with respect to this subsequence in the boundary value problems of $\hat{p}^n$, we easily see that $\hat{p}$ is a classical solution to problem \eqref{eq401}. Therefore, we proved the following lemma.

\begin{lemma}\label{lem402}
Suppose that the S-Condition holds. Then problem \eqref{eq401} has one and only one solution in $C^{k,\alpha}(\overline{\mathcal{M}})$, and it satisfies the estimate \eqref{eq411}.
\end{lemma}

\section{Stability of transonic shock solution}\label{sec5}
We now use a Banach fixed-point theorem to solve the transonic shock problem (T4), provided that the background solution $U_b$ satisfies the S-Condition.

\subsection{The iteration sets}\label{sec41}
Let $\sigma_0$  be a positive constant to be specified later, and
\begin{eqnarray*} \label{258}
\mathcal{K}_\sigma\triangleq\left\{\psi\in C^{4,\alpha}({\T}^2)\,:\,
\norm{\psi-r_b}_{C^{4,\alpha}({\T}^2)}\le
\sigma\le\sigma_0\right\}
\end{eqnarray*}
be the set of possible shock-front.  For any given $\psi\in \mathcal{K}_\sigma$, its position $r^p$ and
profile $\psi^p$ also satisfy
\begin{eqnarray*}
|{r^p-r_b}|\le\sigma,\qquad \norm{\psi^p}_{C^{4,\alpha}(\mathbf\T^2)}\le 2\sigma.
\end{eqnarray*}
We write the set of possible variations of the subsonic flows as
\begin{eqnarray*}
\mathcal{X}_\delta\triangleq\Big\{\check{U}=(\check{p}, \check{s},\check{E}, \check{u}')\,:\,
\norm{\check{U}}_{3}+\norm{i^*\check{U}}_{C^{3,\alpha}(\T^2)}\le \delta\le\delta_0\Big\}.
\end{eqnarray*}
The constants $\sigma_0, \delta_0$ will be chosen later. For $k=2,3$, the norm $\norm{\cdot}_k$ appeared here is defined by
\begin{eqnarray*}
\norm{\check{U}}_k\triangleq \norm{\check{p}}_{C^{k,\alpha}(\overline{\mathcal{M}})}+
\norm{\check{s}}_{C^{k-1,\alpha}(\overline{\mathcal{M}})}+
\norm{\check{E}}_{C^{k-1,\alpha}(\overline{\mathcal{M}})}+
\sum_{\beta=1}^2\norm{\check{u}^\beta}_{C^{k-1,\alpha}(\overline{\mathcal{M}})}.
\end{eqnarray*}
For any $\psi\in \mathcal{K}_\sigma$ and $\check{U}\in
\mathcal{X}_\delta$, we set $$U=\check{U}+U_b^+\left(\frac{L-\psi(y')}{L-r_b}(y^0-L)+L, y'\right).$$

\subsection{Construction of iteration mapping}\label{sec42}
Given $U^-$ satisfying \eqref{eq217}, for any $\psi\in \mathcal{K}_\sigma$ and $\check{U}\in
\mathcal{X}_\delta$, we construct a mapping
$$\mathcal{T}: \mathcal{K}_\sigma\times
\mathcal{X}_\delta\rightarrow \mathcal{K}_\sigma\times \mathcal{X}_\delta, \quad (\psi,\check{U})\mapsto (\hat{\psi},\hat{U})
$$
as follows. One should note that a fixed-point of this mapping is a solution to Problem (T4). We also use $C$ to denote generic positive constants which might be different in different places.

\medskip
\paragraph{\bf Pressure $p$.\ }
We first consider the problem \eqref{eq389} on $\hat{p}$:
\begin{eqnarray}
\begin{cases}\label{eq501}
\mathfrak{L}(\hat{p})=e_6(y^0)\hat{E}^-+F(U,\psi, DU,D^2p, D\psi, D^2\psi)& \text{in}\quad \mathcal{M},\\[3pt]
\hat{p}=p_1-p_b^+ &\text{on}\quad \mathcal{M}_1,\\[3pt]
\Delta'(i^*\hat{p})+\mu_7(i^*{\hat{p}})+\mu_8(i^*\p_0\hat{p})\\[3pt]
\qquad=\bar{g}_8(U,U^-,\psi,DU,D\psi,D^2U,D^2\psi,D^3\psi)&\text{on}\quad \mathcal{M}_0.
\end{cases}
\end{eqnarray}

Here the non-homogeneous terms $F$ and $\bar{g}_8$  are determined by $\psi\in \mathcal{K}_\sigma$ and $U=\check{U}+U_b^+$, with $\check{U}\in X_\delta$, and $\hat{E}^-$ is solved from \eqref{eq217}.
Then, since we assumed that the S-Condition holds, by Lemma \ref{lem402}, we could solve uniquely one $\hat{p}\in C^{3,\alpha}(\overline{\mathcal{M}})$ and it satisfies the following estimate:
\begin{eqnarray}\label{eq502}
\norm{\hat{p}}_{C^{3,\alpha}(\overline{\mathcal{M}})}\le C\Big(\norm{\hat{E}^-}_{C^{1,\alpha}(\overline{\mathcal{M}})}+\norm{F}_{C^{1,\alpha}(\overline{\mathcal{M}})}
+\norm{p_1-p_b^+}_{C^{3,\alpha}(\mathcal{M}_1)}+\norm{\bar{g}_8}_{C^{1,\alpha}(\T^2)}
\Big).
\end{eqnarray}
Checking the definitions of $F$ and $\bar{g}_8$,  we have
\begin{eqnarray}\label{eq503}
\norm{F}_{C^{1,\alpha}(\overline{\mathcal{M}})}\le C(\delta^2+\sigma^2+\varepsilon^2+\varepsilon), \qquad \norm{\bar{g}_8}_{C^{1,\alpha}(\T^2)}\le C(\delta^2+\varepsilon^2+\sigma^2+\varepsilon).
\end{eqnarray}
So combining \eqref{eq217}, \eqref{eq210} and \eqref{eq502}, \eqref{eq503}, one infers that
\begin{eqnarray}\label{eq504}
\norm{\hat{p}}_{C^{3,\alpha}(\overline{\mathcal{M}})}\le C\Big(\delta^2+\sigma^2+\varepsilon^2+\varepsilon\Big).
\end{eqnarray}

\medskip
\paragraph{\bf Update shock-front $\hat{\psi}$.\ }
With the specified higher-order terms $\bar{g}_5$ and $\bar{g}_7$, and $\hat{p}$ solved from \eqref{eq501}, we now set ({\it cf.} \eqref{eq366})
\begin{eqnarray}
&&\begin{cases}\label{eq505}\displaystyle
\hat{r}^p-r_b=-\frac{1}{4\pi^2\mu_6}\int_{{\T}^2}\Big(\mu_5\,
i^*(\p_0\hat{p})+\bar{g}_5(U,U^-,\psi,DU,D\psi)\Big)\, \dd x^1\dd x^2,\\[3pt]
\displaystyle\hat{\psi}^p=\frac{1}{\mu_2}\left(i^*(\hat{p})-{\mu_9}
\int_{{\T}^2}i^*(\p_0\hat{p})\,\dd x^1\dd x^2+{\bar{g}_9(U,U^-,\psi,D\psi)}\right),\\[3pt]
\hat{\psi}=\hat{\psi}^p+\hat{r}^p.
\end{cases}
\end{eqnarray}

It follows easily that (using \eqref{eq504})
\begin{align}\label{eq506}
\norm{\hat{\psi}^p}_{C^0(\T^2)}+|\hat{r}^p-r_b|\le& C\Big(\norm{\bar{g}_5}_{C^0(\T^2)}+\norm{\bar{g}_9}_{C^0(\T^2)}
+\norm{\hat{p}}_{C^{1}(\overline{\mathcal{M}})}\Big)\nonumber\\
\le&
C\Big(\delta^2+\sigma^2+\varepsilon^2+\varepsilon\Big).
\end{align}

For the $C^{4,\alpha}$ estimate of $\hat{\psi}^p$, we note that $i^*\hat{p}$ solves the third  equation in \eqref{eq501}, hence $\hat{\psi}^p$ solves the following elliptic equation on $\T^2$ ({\it cf.} \eqref{eq351}):
\begin{eqnarray}\label{eq507}
\Delta'\hat{\psi}^p+\mu_7\hat{\psi}^p=\mu_0\mu_6(\hat{r}^p-r_b)+\mu_0\mu_5\,i^*(\p_0\hat{p})
+\bar{g}_6(U,U^-,\psi,DU^-,DU,D\psi,D^2\psi).
\end{eqnarray}
Standard Schauder estimates \cite[Chapter 6]{GT2001} yield that
\begin{align}\label{eq508}
\norm{\hat{\psi}^p}_{C^{4,\alpha}(\T^2)}\le& C\Big(\norm{\hat{\psi}^p}_{C^0(\T^2)}
+|\hat{r}^p-r_b|+\norm{\hat{p}}_{C^{3,\alpha}(\T^2)}+\norm{\bar{g}_6}_{C^{2,\alpha}(\T^2)}
\Big)\nonumber\\
\le&C\Big(\delta^2+\sigma^2+\varepsilon^2+\varepsilon\Big).
\end{align}
Hence one has
\begin{eqnarray}\label{eq509}
\norm{\hat{\psi}-r_b}_{C^{4,\alpha}(\T^2)}\le C_0\Big(\delta^2+\sigma^2+\varepsilon^2+\varepsilon\Big).
\end{eqnarray}

We also need to show that
\begin{eqnarray}\label{eq510}
\int_{\T^2}\hat{\psi}^p\,\dd x^1\dd x^2=0.
\end{eqnarray}
In fact, integrating  \eqref{eq507} on $\T^2$, and recall $\bar{g}_6=\mu_0\bar{g}_5+\p_{\beta}\bar{g_0}^\beta$, using divergence theorem and definition of $\hat{r}^p-r_b$ in \eqref{eq505}, we  have directly \eqref{eq510}.

\medskip
\paragraph{\bf Entropy $A(s)$.\ }
Note that ${A(s_b^+)}$ is constant, we  solve the problem \eqref{eq367}
\begin{eqnarray}\label{eq511}
\begin{cases}
D_u\widehat{A(s)}=0&\text{in}\quad \mathcal{M},\\
i^*(\widehat{A(s)})=\mu_4\,(\hat{\psi}-r_b)+\bar{g}_4(U,U^-,\psi,D\psi)&\text{on}\quad \mathcal{M}_0
\end{cases}
\end{eqnarray}
to obtain the unique solution $\widehat{A(s)}$. It also holds
\begin{align}
\norm{\widehat{A(s)}}_{C^{2,\alpha}(\overline{\mathcal{M}})}\le& C\norm{i^*\widehat{A(s)}}_{C^{3,\alpha}(\T^2)}\nonumber\\
\le&C\Big(\norm{\hat{\psi}-r_b}_{C^{3,\alpha}(\T^2)}+\norm{\bar{g}_4}_{C^{3,\alpha}(\T^2)}\Big)\nonumber\\
\le&C\Big(\delta^2+\sigma^2+\varepsilon^2+\varepsilon\Big).\label{eq512}
\end{align}

\medskip

\paragraph{\bf Bernoulli constant $E$.\ }
We then solve the linear problem \eqref{eq364} on $\hat{E}$:
\begin{eqnarray}\label{eq513}
\begin{cases}
D_u \hat{E}+2\mu u^0 \hat{E}=\frac{2\mu u^0 }{\gamma-1}\rho_{b}^{\gamma-1}\widehat{A(s)}+\frac{2\mu u^0}{\rho_b} \hat{p}+\overline{H}(U, \psi, D\psi)&\text{in}\quad \mathcal{M},\\[3pt]
\hat{E}=E^--E_b^-&\text{on}\quad \mathcal{M}_0.
\end{cases}
\end{eqnarray}
Hence we could easily get the unique existence of $\hat{E}\in C^{2,\alpha}(\overline{\mathcal{M}})$ (note that $u\in C^{2,\alpha}(\overline{\mathcal{M}})$) with
\begin{align}\label{eq514}
\norm{\hat{E}}_{C^{2,\alpha}(\overline{\mathcal{M}})}\le&C\left(\norm{i^*(E^--E_b^-)}_{C^{3,\alpha}(\T^2)}
+\norm{\widehat{A(s)}}_{C^{2,\alpha}
		(\overline{\mathcal{M}})}+\norm{\hat{p}}_{C^{2,\alpha}
		(\overline{\mathcal{M}})}+\norm{\overline{H}}_{C^{2,\alpha}
		(\overline{\mathcal{M}})}\right)\nonumber\\
\le& C\Big(\delta^2+\sigma^2+\varepsilon^2+\varepsilon\Big).
\end{align}
The estimate \eqref{eq217} is used to obtain the second inequality.

\medskip
\paragraph{\bf Tangential velocity field $\bar{u}_0'$ on $\mathcal{M}_0$.}
Next we solve tangential velocity $\bar{u}'_0$ on $\mathcal{M}_0$ from ({\it cf.} \eqref{eq368})
\begin{eqnarray}\label{eq515}
\begin{cases}
\p_2({\bar{u}}^1_0)-\p_1({\bar{u}}^2_0)=-\frac{1}{\mu_0}\Big(\p_2\bar{g}_0^1(U, U^-,D\psi)-\p_1\bar{g}_0^2(U, U^-,D\psi)\Big)\\[3pt]
\p_\beta(\bar{u}^\beta_0)=\mu_5\
i^*(\p_0\hat{p})+\mu_6\,\hat{\psi}^p+\mu_6\,(\hat{r}^p-r_b)+\bar{g}_5(U,U^-,\psi,DU,D\psi)\\[3pt]
\int_{0}^{2\pi}\Big(\mu_0\bar{u}^1_0+\bar{g}_0^1\Big)(s,\pi)\,\dd s=0\\[3pt]
\int_{0}^{2\pi}\Big(\mu_0\bar{u}^2_0+\bar{g}_0^2\Big)(\pi,s)\,\dd s=0
\end{cases} \text{on}\, \T^2.
\end{eqnarray}
By applications of de Rham's Theorem and Hodge Theorem, or treated as in \cite[pp.546-547]{ChenYuan2008}, one can solve a unique  $\bar{u}_0'$ on $\T^2$, and the  following estimate is valid:
\begin{align}\label{eq516}
\norm{\bar{u}_0'}_{C^{3,\alpha}(\T^2)}\le& C\Big(\sum_{\beta=1}^2\norm{\bar{g}^\beta_0}_{C^{3,\alpha}(\T^2)}
+\norm{\hat{p}}_{C^{3,\alpha}(\overline{\mathcal{M}})}
+\norm{\hat{\psi}-r_b}_{C^{2,\alpha}(\T^2)}
+\norm{\bar{g}_5}_{C^{2,\alpha}(\T^2)}\Big)\nonumber\\
\le& C \Big(\delta^2+\sigma^2+\varepsilon^2+\varepsilon\Big).
\end{align}

\medskip

\paragraph{\bf Tangential velocity $\hat{u}'$ in $\mathcal{M}$.}
Finally, we solve the tangential velocity $\hat{u}'$ in $\mathcal{M}$ through ({\it cf.} \eqref{eq369})
\begin{eqnarray}
\begin{cases}\label{eq517}
D_u\hat{u}^\beta=-\frac{1}{\rho}\p_\beta \hat{p}+\overline{W}_{\beta}(U,\psi,Dp,D\psi) &\text{in}\quad \mathcal{M},\\[3pt]
\hat{u}^\beta=\hat{u}_0^\beta&\text{on}\quad \mathcal{M}_0.
\end{cases}
\end{eqnarray}
Here the Cauchy data $\hat{u}_0^\beta$ ($\beta=1,2$) on $\mathcal{M}_0$ is solved from \eqref{eq515}.

From \eqref{eq504} and \eqref{eq516}, we obtain a unique $\hat{u}^\beta$ in $\mathcal{M}$  and it holds that
\begin{align}\label{eq518}
\norm{\hat{u}^\beta}_{C^{2,\alpha}(\overline{\mathcal{M}})}\le& C\Big(\norm{\hat{u}_0^\beta}_{C^{2,\alpha}(\mathcal{M}_0)}
+\norm{\hat{p}}_{C^{3,\alpha}(\overline{\mathcal{M}})}
+\norm{\overline{W}_{\beta}}_{C^{2,\alpha}(\overline{\mathcal{M}})}\Big)\nonumber\\
\le&C\Big(\delta^2+\sigma^2+\varepsilon^2+\varepsilon\Big).
\end{align}

\medskip
\paragraph{\bf Conclusion.\ } From the above six steps, we get uniquely one pair $(\hat{U}, \hat{\psi})$ and it follows from \eqref{eq504}, \eqref{eq509}, \eqref{eq512}, \eqref{eq514}, \eqref{eq516},
and \eqref{eq518} that
\begin{eqnarray}\label{eq519}
\norm{\hat{\psi}-r_b}_{C^{4,\alpha}(\T^2)}+\norm{\hat{U}}_3+\norm{i^*\hat{U}}_{C^{3,\alpha}(\T^2)}\le \tilde{C}\Big(\delta^2+\sigma^2+\varepsilon^2+\varepsilon\Big).
\end{eqnarray}
Here $\tilde{C}$ is a constant depending only on the background solution and $L, \alpha$.
Now we choose $C_*=4\tilde{C}$ and $\varepsilon_0\le \min\Big\{ {1}/{(16\tilde{C}^2)}, 1, h_b/(8\tilde{C})\Big\}$. Then, for $\delta=\sigma=C_*\varepsilon$, we have $\tilde{C}\Big(\delta^2+\sigma^2+\varepsilon^2+\varepsilon\Big)\le\delta,\,
\forall \varepsilon\in(0,\varepsilon_0),$ and the  estimate \eqref{eq519} shows that $\hat{\psi}\in \mathcal{K}_{C_*\varepsilon}$ and $\hat{U}\in \mathcal{X}_{C_*\varepsilon}$. Hence we construct the desired mapping $\mathcal{T}$ on $\mathcal{K}_{C_*\varepsilon}\times \mathcal{X}_{C_*\varepsilon}$.

\subsection{Contraction of iteration mapping}\label{sec43}
What left is to show that the mapping $$\mathcal{T}: \mathcal{K}_{C_*\varepsilon}\times
\mathcal{X}_{C_*\varepsilon}\rightarrow \mathcal{K}_{C_*\varepsilon}\times \mathcal{X}_{C_*\varepsilon}, \quad (\psi,\check{U})\mapsto (\hat{\psi},\hat{U})
$$
is a contraction in the sense that
\begin{align}\label{eq520}
&\norm{\hat{\psi}^{(1)}-\hat{\psi}^{(2)}}_{C^{3,\alpha}(\T^2)}+\norm{\hat{U}^{(1)}
-\hat{U}^{(2)}}_2+\norm{i^*(\hat{U}^{(1)}
-\hat{U}^{(2)})}_{C^{2,\alpha}(\T^2)}\nonumber\\
\le&\frac12\Big(\norm{{\psi}^{(1)}-{\psi}^{(2)}}_{C^{3,\alpha}(\T^2)}+\norm{\check{U}^{(1)}
-\check{U}^{(2)}}_2+\norm{i^*(\check{U}^{(1)}
-\check{U}^{(2)})}_{C^{2,\alpha}(\T^2)}\Big)\nonumber\\
\triangleq&\frac12\mathcal{Q},
\end{align}
provided that $\varepsilon_0$ is further small (depending only on the background solution and $L, \alpha$). Here for $j=1,2$, and any $\psi^{(j)}\in \mathcal{K}_{{C_*\varepsilon}}, \check{U}^{(j)}\in \mathcal{X}_{{C_*\varepsilon}}$, we have defined $(\hat{\psi}^{(j)}, \hat{U}^{(j)})=\mathcal{T}(\psi^{(j)}, \check{U}^{(j)})$.

To prove \eqref{eq520}, we set $\widetilde{\psi}=\hat{\psi}^{(1)}-\hat{\psi}^{(2)}$, and $\widetilde{U}=\hat{U}^{(1)}-\hat{U}^{(2)}$.   For $k=1,2$, we also use the notations
\begin{align*}
(U^-)^{(k)}=&\left.U^-\right|_{S^{\psi^{(k)}}},\quad (\hat{U}^-)^{(k)}= (U^--U_b^-)^{(k)},\quad
i^*(U_b^+)^{(k)}=\left.U_b^+\right|_{S^{\psi^{(k)}}},\\ (U_b^+)^{(k)}=&(U_b^+)\Big(\frac{L-\psi^{(k)}(y')}{L-r_b}(y^0-L)+L, y'\Big),
\quad U^{(k)}=\check{U}^{(k)}+(U_b^+)^{(k)}.
\end{align*}
By \eqref{eq517} and analyticity of $U_b^\pm$, the mean value theorem implies that
\begin{eqnarray}
&&\norm{(\hat{U}^-)^{(1)}-(\hat{U}^-)^{(2)}}_{C^{2,\alpha}(\T^2)}\le C\varepsilon\norm{\psi^{(1)}-\psi^{(2)}}_{C^{2,\alpha}(\T^2)},\label{eq521}\\
&&\norm{i^*(U_b^+)^{(1)}-i^*(U_b^+)^{(2)}}_{C^{k,\alpha}(\T^2)}\le C\norm{\psi^{(1)}-\psi^{(2)}}_{C^{k,\alpha}(\T^2)},\label{eq522}\\
&&\norm{(U_b^+)^{(1)}-(U_b^+)^{(2)}}_{C^{k,\alpha}{(\overline{\mathcal{M}}})}\le C\norm{\psi^{(1)}-\psi^{(2)}}_{C^{k,\alpha}(\T^2)},\quad k=1,2,3,4.\label{eq523}
\end{eqnarray}

\noindent {\it Step 1.} Firstly we seek an estimate of $\widetilde{p}$, which solves  ({\it cf.} \eqref{eq501})
\begin{eqnarray*}
\begin{cases}
\mathfrak{L}(\widetilde{p})=e_6(y^0)\Big((\hat{E}^-)^{(1)}-(\hat{E}^-)^{(2)}\Big)+F^{(1)}-F^{(2)}& \text{in}\quad \mathcal{M},\\[3pt]
\widetilde{p}=0 &\text{on}\quad \mathcal{M}_1,\\[3pt]
\Delta'(i^*\widetilde{p})+\mu_7(i^*\widetilde{p})+\mu_8(i^*\p_0\widetilde{p})
=\bar{g}_8^{(1)}-\bar{g}_8^{(2)}&\text{on}\quad \mathcal{M}_0.
\end{cases}
\end{eqnarray*}
Here for $k=1,2$,
\begin{eqnarray*}
&&F^{(k)}=f(U^{(k)},\psi^{(k)}, DU^{(k)},D^2p^{(k)}, D\psi^{(k)}, D^2\psi^{(k)}),\\
&&\bar{g}_8^{(k)}=
\bar{g}_8(U^{(k)},(U^-)^{(k)},\psi^{(k)},DU^{(k)},D\psi^{(k)},D^2U^{(k)},D^2\psi^{(k)},D^3\psi^{(k)}).
\end{eqnarray*}
By Lemma \ref{lem401} and \eqref{eq521}, direct computation yields
\begin{align}\label{eq524}
\norm{\widetilde{p}}_{C^{2,\alpha}(\overline{\mathcal{M}})}\le& C\left(\norm{(\hat{E}^-)^{(1)}-(\hat{E}^-)^{(2)}}_{C^{\alpha}(\overline{\mathcal{M}})}
+\norm{F^{(1)}-F^{(2)}}_{C^{\alpha}(\overline{\mathcal{M}})}
+\norm{\bar{g}_8^{(1)}-\bar{g}_8^{(2)}}_{C^\alpha(\T^2)}\right)\nonumber\\
\le& C\varepsilon \mathcal{Q}.
\end{align}

\noindent {\it Step 2.} From \eqref{eq505}, we see that
\begin{eqnarray*}
&&\begin{cases}\displaystyle
\tilde{r}^p=-\frac{1}{4\pi^2\mu_6}\int_{\T^2}\Big(\mu_5\,
i^*(\p_0\widetilde{p})+(\bar{g}_5^{(1)}-\bar{g}_5^{(2)})\Big)\, \dd x^1\dd x^2,\\[3pt]
\widetilde{\psi}^p=\frac{1}{\mu_2}\left(i^*\widetilde{p}-{\mu_9}
\int_{\T^2}i^*(\p_0\widetilde{p})\,\dd x^1\dd x^2+(\bar{g}_9^{(1)}-\bar{g}_9^{(2)})\right),\\[3pt]
\widetilde{\psi}=\widetilde{\psi}^p+\tilde{r}^p.
\end{cases}
\end{eqnarray*}
Here, for $k=1,2$,
\begin{eqnarray*}
&&\bar{g}_5^{(k)}=\bar{g}_5(U^{(k)},(U^-)^{(k)},\psi^{(k)},DU^{(k)},D\psi^{(k)}),\\
&&\bar{g}_7^{(k)}=\bar{g}_7(U^{(k)},(U^-)^{(k)},\psi^{(k)},D\psi^{(k)}).
\end{eqnarray*}
Then we have the following estimate via \eqref{eq524}, and some straightforward computations:
\begin{align}\label{eq525}
\norm{\widetilde{\psi}^p}_{C(\T^2)}+|\tilde{r}^p|\le& C\left(\norm{\widetilde{p}}_{C^1(\overline{\mathcal{M}})}+\norm{\bar{g}_5^{(1)}-\bar{g}_5^{(2)}}_{C(\overline{\mathcal{M}})}+
\norm{\bar{g}_7^{(1)}-\bar{g}_7^{(2)}}_{C(\overline{\mathcal{M}})}\right)\nonumber\\
\le&C\varepsilon \mathcal{Q}.
\end{align}
By \eqref{eq507}, note that $\widetilde{\psi}^p$ also solves
\begin{eqnarray*}
\Delta'\widetilde{\psi}^p+\mu_7\widetilde{\psi}^p=\mu_0\mu_6\tilde{r}^p+\mu_0\mu_5\,i^*\p_0\widetilde{p}
+\bar{g}_6^{(1)}-\bar{g}_6^{(2)},
\end{eqnarray*}
with
\begin{eqnarray*}
\bar{g}_6^{(k)}=\bar{g}_6(U^{(k)},(U^-)^{(k)},\psi^{(k)},(DU^-)^{(k)},DU^{(k)},D\psi^{(k)},D^2\psi^{(k)}),\quad k=1,2,
\end{eqnarray*}
it follows that, from \eqref{eq524} and \eqref{eq525},
\begin{align*}
\norm{\widetilde{\psi}^p}_{C^{3,\alpha}(\T^2)}\le& C\left(\norm{\widetilde{\psi}^p}_{C(\T^2)}+|\tilde{r}^p|+
\norm{\widetilde{p}}_{C^{2,\alpha}(\overline{\mathcal{M}})}
+\norm{\bar{g}_6^{(1)}-\bar{g}_6^{(2)}}_{C^{1,\alpha}(\T^2)}\right)\nonumber\\
\le&C\varepsilon \mathcal{Q}.
\end{align*}
This and \eqref{eq525} imply that
\begin{eqnarray}\label{eq526}
\norm{\widetilde{\psi}}_{C^{3,\alpha}(\T^2)}\le C\varepsilon \mathcal{Q}.
\end{eqnarray}

\noindent {\it Step 3.} From \eqref{eq511}, one has
\begin{eqnarray*}
\begin{cases}
D_{u^{(1)}}\widetilde{A(s)}+D_{u^{(1)}-u^{(2)}}\widehat{A(s)}^{(2)}=0&\text{in}\quad \mathcal{M},\\[3pt]
i^*(\widetilde{A(s)})=\mu_4\,\widetilde{\psi}+\bar{g}_4^{(1)}-\bar{g}_4^{(2)}&\text{on}\quad \mathcal{M}_0,
\end{cases}
\end{eqnarray*}
where
\begin{eqnarray*}
\bar{g}^{(k)}_4=\bar{g}_4(U^{(k)},(U^-)^{(k)},\psi^{(k)},D\psi^{(k)}),\quad k=1, 2.
\end{eqnarray*}
By \eqref{eq523} and \eqref{eq526}, we have
\begin{align}\label{eq527}
\norm{\widetilde{A(s)}}_{C^{1,\alpha}(\overline{\mathcal{M}})}\le& C\left( \norm{i^*(\widetilde{A(s)})}_{C^{2,\alpha}(\T^2)}+\norm{u^{(1)}-u^{(2)}}_{C^{1,\alpha}(\overline{\mathcal{M}})}
\norm{\widehat{A(s)}^{(2)}}_{C^{2,\alpha}(\overline{\mathcal{M}})}\right)\nonumber\\
\le&C\left\{\norm{\widetilde{\psi}}_{C^{2,\alpha}(\T^2)}
+\norm{\bar{g}_4^{(1)}-\bar{g}_4^{(2)}}_{C^{2,\alpha}(\T^2)}\right.\nonumber\\
&\left.
+\, \varepsilon\left(\norm{\check{U}^{(1)}-\check{U}^{(2)}}_{C^{1,\alpha}(\overline{\mathcal{M}})}
+\norm{{U_b^+}^{(1)}-{U_b^+}^{(2)}}_{C^{1,\alpha}(\overline{\mathcal{M}})}\right)\right\}\nonumber\\
\le& C\varepsilon \mathcal{Q}.
\end{align}

\noindent{\it Step 4.} We note that $\widetilde{E}$ solves the following problem ({\it cf.} \eqref{eq513})
\begin{eqnarray*}
\begin{cases}
\frac{1}{(u^0)^{(1)}}D_{u^{(1)}}\widetilde{E}+2\mu \widetilde{E}+\left(\frac{1}{(u^0)^{(1)}}D_{u^{(1)}}-\frac{1}{(u^0)^{(2)}}D_{u^{(2)}}\right)\hat{E}^{(2)}\\[3pt]
\qquad=\frac{2\mu }{\gamma-1}\rho_{b}^{\gamma-1}\widetilde{A(s)}+\frac{2\mu }{\rho_b} \widetilde{p}+\frac{1}{(u^0)^{(1)}}\overline{H}^{(1)}-\frac{1}{(u^0)^{(2)}}\overline{H}^{(2)}&\text{in}\ \ \mathcal{M},\\[3pt]
\widetilde{E}=(\hat{E}^-)^{(1)}-(\hat{E}^-)^{(2)} &\text{on}\ \ \mathcal{M}_0.
\end{cases}
\end{eqnarray*}
with
\begin{eqnarray*}
\bar{H}^{(k)}=\bar{H}(U^{(k)},\psi^{(k)},D\psi^{(k)}),\quad k=1,2.
\end{eqnarray*}
Then using  \eqref{eq521}, \eqref{eq523}, \eqref{eq524} and \eqref{eq527}, one has
\begin{align}\label{eq528}
\norm{\widetilde{E}}_{C^{1,\alpha}(\overline{\mathcal{M}})}\le&
C\left(\norm{(\hat{E}^-)^{(1)}-(\hat{E}^-)^{(2)}}_{C^{1,\alpha}(\T^2)}
+\norm{u^{(1)}-u^{(2)}}_{C^{1,\alpha}(\overline{\mathcal{M}})}
\norm{\hat{E}^{(2)}}_{C^{2,\alpha}(\overline{\mathcal{M}})}
\right.\nonumber\\
&\quad\left. +\norm{\widetilde{A(s)}}_{C^{1,\alpha}(\overline{\mathcal{M}})}
+\norm{\widetilde{p}}_{C^{1,\alpha}(\overline{\mathcal{M}})}
+\norm{\frac{1}{(u^0)^{(1)}}\bar{H}^{(1)}-\frac{1}{(u^0)^{(2)}}\bar{H}^{(2)}}_{C^{1,\alpha}
(\overline{\mathcal{M}})}\right)\nonumber\\
\le& C\varepsilon\mathcal{Q}.
\end{align}

\noindent {\it Step 5.} Next by \eqref{eq515} we find that the difference of tangential velocity field on $\mathcal{M}_0$ solves
\begin{eqnarray*}
\begin{cases}
\p_2({\tilde{u}}^1_0)-\p_1({\tilde{u}}^2_0)=-\frac{1}{\mu_0}\Big(\p_2\big((\bar{g}_0^{1})^{(1)}
-(\bar{g}_0^{1})^{(2)}\big)
-\p_1\big((\bar{g}_0^{2})^{(1)}-(\bar{g}_0^{2})^{(2)}\big)\Big),\\[3pt]
\p_\beta(\tilde{u}^\beta_0)=\mu_5\,
i^*(\p_0\widetilde{p})+\mu_6\,
\widetilde{\psi}+(\bar{g}_5^{(1)}-\bar{g}_5^{(2)}),\\[3pt]
\int_{0}^{2\pi}\Big(\mu_0\tilde{u}^1_0+\big((\bar{g}_0^{1})^{(1)}
-(\bar{g}_0^{1})^{(2)}\big)\Big)(s,x^2)ds=0,\\[3pt]
\int_{0}^{2\pi}\Big(\mu_0\tilde{u}^2_0+\big((\bar{g}_0^{2})^{(1)}-(\bar{g}_0^{2})^{(2)}\big)\Big)(x^1,s)ds=0,
\end{cases}\qquad \text{on}\  \T^2;
\end{eqnarray*}
where, for $k,\beta=1,2$,
\begin{align*}
(\bar{g}_0^{\beta})^{(k)}=&\bar{g}_0^{\beta}(U^{(k)},(U^-)^{(k)}, \psi^{(k)}, D\psi^{(k)}),\\
\bar{g}_5^{(k)}=&\bar{g}_5(U^{(k)},(U^-)^{(k)},\psi^{(k)},DU^{(k)},D\psi^{(k)}).
\end{align*}
We easily deduce the estimate
\begin{align}\label{eq529}
\norm{{\tilde{u}}_0'}_{C^{2,\alpha}(\T^2)}\le& C\left(\sum_{\beta=1}^2\norm{(\bar{g}_0^{\beta})^{(1)}-(\bar{g}_0^{\beta})^{(2)}}_{C^{2,\alpha}(\T^2)}+\norm{\widetilde{p}}_{C^{2,\alpha}(\overline{\mathcal{M}})}
+\norm{\widetilde{\psi}}_{C^{1,\alpha}(\T^2)}\right.\nonumber\\
&\left.\quad+\norm{\bar{g}_5^{(1)}-\bar{g}_5^{(2)}}_{C^{1,\alpha}(\T^2)}\right)\nonumber\\
\le&C\varepsilon \mathcal{Q}.
\end{align}

\noindent {\it Step 6.} From \eqref{eq517}, $\tilde{u}^\beta$, ($\beta=1,2$) solves
\begin{eqnarray*}
\begin{cases}
(u^j)^{(1)}\p_j\tilde{u}^\beta+\Big((u^j)^{(1)}
-(u^j)^{(2)}\Big)\p_j(\hat{u}^{(\beta)})\\[3pt]
\qquad=-\frac{1}{\rho^{(1)}}\p_\beta \tilde{p}+\left(\frac{1}{\rho^{(2)}}-\frac{1}{\rho^{(1)}}\right)\p_\beta \hat{p}^{(2)} +(\overline{W}_\beta^{(1)}-\overline{W}_\beta^{(2)})&\text{in}\quad \mathcal{M},\\[3pt]
\tilde{u}^\beta=\tilde{u}_0^\beta&\text{on}\quad \mathcal{M}_0,
\end{cases}
\end{eqnarray*}
with
\begin{eqnarray*}
\overline{W}_\beta^{(k)}=\overline{W}_\beta(U^{(k)}, \psi^{(k)}, Dp^{(k)}, D\psi^{(k)}),\quad k=1,2.
\end{eqnarray*}
So there holds
\begin{align}\label{eq530}
\norm{\tilde{u}^\beta}_{C^{1,\alpha}(\overline{\mathcal{M}})}\le& C\left(\norm{{\tilde{u}}_0^\beta}_{C^{1,\alpha}(\T^2)}+\norm{\tilde{p}}_{C^{2,\alpha}(\overline{\mathcal{M}})}
+\varepsilon\norm{U^{(1)}-U^{(2)}}_{C^{1,\alpha}(\overline{\mathcal{M}})}\right.\nonumber\\
&\left.\quad+\norm{\overline{W}_\beta^{(1)}-\overline{W}_\beta^{(2)}}_{C^{1,\alpha}(\overline{\mathcal{M}})}\right)\nonumber\\
\le&C\varepsilon \mathcal{Q}.
\end{align}

\noindent {\it Conclusion.} Now summing up the inequalities \eqref{eq524}--\eqref{eq530}, we get
\begin{eqnarray*}
\norm{\hat{\psi}}_{C^{3,\alpha}(\T^2)}+\norm{\hat{U}}_2
+\norm{i^*\hat{U}}_{C^{2,\alpha}(\T^2)}
\le C'\varepsilon \mathcal{Q},
\end{eqnarray*}
which implies \eqref{eq520} if $\varepsilon\in(0, \varepsilon_0)$ and  $C'\varepsilon_0<1/2$.
Finally, by a Banach fixed-point theorem, we infer Problem (T4), hence Problem (T), has one and only one solution in $\mathcal{K}_{C_*\varepsilon}\times \mathcal{X}_{C_*\varepsilon}.$ This finishes the proof of Theorem  \ref{thm201}.

\bigskip
{\bf Acknowledgments.}
The authors thank sincerely an anonymous reader for pointing out a serious mistake on Lemma \ref{lem401} in a previous version of this manuscript.  We reformulated Lemma \ref{lem401} and presented a detailed proof in this new version.  This work is supported  by National Nature Science Foundation of China under
Grant No. 11371141 and No. 11871218; by Science and Technology Commission of Shanghai Municipality (STCSM) under Grant No. 18dz2271000.



\begin{thebibliography}{99}
\bibitem{AN-2000}
D. E. Apushkinskaya;  A. I. Nazarov.
A survey of results on nonlinear Venttsel problems.
Appl. Math.  45  (2000),  no. 1, 69--80.


\bibitem{BaeFeldman2011}
M. Bae; M. Feldman. Transonic shocks in multidimensional divergent nozzles. Arch. Ration. Mech. Anal.  201  (2011),  no. 3, 777--840.


\bibitem{BS2007}
S. Benzoni-Gavage; D. Serre. { Multidimensional Hyperbolic
Partial Differential Equations: First-order Systems and
Applications.} Oxford Mathematical Monographs. Clarendon Press,
Oxford, 2007.

\bibitem{ChenXie2014}
C. Chen; C. Xie. Three dimensional steady subsonic Euler flows in bounded nozzles. J. Differential Equations  256  (2014),  no. 11, 3684--3708.

\bibitem{chen-feldman-2003}
G.-Q. Chen; M. Feldman. Multidimensional transonic shocks and free boundary problems for nonlinear equations of mixed type. J. Amer. Math. Soc.  16  (2003),  no. 3, 461--494.


\bibitem{ChenYuan2013}
G.-Q. Chen; H. Yuan. Local uniqueness of steady spherical transonic shock-fronts for the three-dimensional full Euler equations. Commun. Pure Appl. Anal. 12 (2013), no. 6, 2515--2542.


\bibitem{ChenYuan2008}
S. Chen; H. Yuan. Transonic shocks in compressible flow passing a duct for three-dimensional Euler systems. Arch. Ration. Mech. Anal.  187  (2008),  no. 3, 523--556.

\bibitem{CHHQ2016}
S.-W. Chou; J. M. Hong; B.-C. Huang; R. Quita. Global transonic solutions to combined Fanno Rayleigh flows through variable nozzles. arXiv:1611.10083, 2016.

\bibitem{CF1976}
R. Courant; K. O. Friedrichs. Supersonic flow and shock waves. Applied Mathematical Sciences, Vol. 21. Springer-Verlag, New York-Heidelberg, 1976.


\bibitem{Dafermos2010}
C.~M. Dafermos. { Hyperbolic Conservation Laws in Continuum Physics}. Third edition. Grundlehren der mathematischen Wissenschaften, vol. 325. Springer-Verlag, Berlin Heidelberg, 2010.


\bibitem{FangLiuYuan2013}
B. Fang; L. Liu; H. Yuan. Global uniqueness of transonic shocks in two-dimensional steady compressible Euler flows. Arch. Ration. Mech. Anal.  207  (2013),  no. 1, 317--345.



\bibitem{GT2001}
D. Gilbarg; N. S. Trudinger. Elliptic partial differential equations of second order. Reprint of the 1998 edition. Classics in Mathematics. Springer-Verlag, Berlin, 2001.

\bibitem{HMP2005}
F. Huang; P. Marcati; R. Pan. Convergence to the Barenblatt solution for the compressible Euler equations with damping and vacuum. Arch. Ration. Mech. Anal.  176  (2005),  no. 1, 1--24.

\bibitem{HuangPanWang2011}
F. Huang; R. Pan; Z. Wang.  $L^1$ convergence to the Barenblatt solution for compressible Euler equations with damping. Arch. Ration. Mech. Anal. 200 (2011), no. 2, 665--689.


\bibitem{LiXinYin2013}
J. Li; Z. Xin; H. Yin. Transonic shocks for the full compressible Euler system in a general two-dimensional de Laval nozzle. Arch. Ration. Mech. Anal.  207  (2013),  no. 2, 533--581.


\bibitem{LiuYuan2008}
L. Liu; H. Yuan. Stability of cylindrical transonic shocks for the two-dimensional steady compressible Euler system. J. Hyperbolic Differ. Equ.  5  (2008),  no. 2, 347--379.


\bibitem{LiuXuYuan2016}
L. Liu; G. Xu; H. Yuan. Stability of spherically symmetric subsonic flows and
transonic shocks under multidimensional perturbations. Adv. Math. 291 (2016), 696--757.

\bibitem{Liu1982-1}
T. P. Liu.
Transonic gas flow in a duct of varying area.
Arch. Rational Mech. Anal.  80  (1982), no. 1, 1--18.

\bibitem{Liu1982-2}
T. P. Liu.   Nonlinear stability and instability of transonic flows through a nozzle. Comm. Math. Phys.  83  (1982), no. 2, 243--260.

\bibitem{LN1991}
Y. Luo; N. S. Trudinger. Linear second order elliptic equations with Venttsel boundary conditions. Proc. Roy. Soc. Edinburgh Sect. A  118  (1991),  no. 3-4, 193--207.

\bibitem{Ra2010}
E. Rathakrishnan. Applied gas dynamics. John Wiley \& Sons (Asia) Pte Ltd, 2010.

\bibitem{RauchXieXin2013}
J. Rauch; C. Xie; Z. Xin. Global stability of steady transonic Euler shocks in quasi-one-dimensional nozzles. J. Math. Pures Appl. (9)  99  (2013),  no. 4, 395--408.

\bibitem{Shapiro1953}
A. H. Shapiro. The dynamics and thermodynamics of compressible fluid flow, Vol. 1, Ronald Press Co., New York, 1953.


\bibitem{Tsuge2015}
N. Tsuge. Existence of global solutions for isentropic gas flow in a divergent nozzle with friction. J. Math. Anal. Appl. 426 (2015), no. 2, 971--977.

\bibitem{vandyke}
M. Van Dyke. An album of fluid motion. The Parabolic Press, Stanford, California, 1982.

\bibitem{Venttsel1959}
A. D. Venttsel. On boundary conditions for multi-dimensional diffusion processes. Theor. Probability Appl. 4 (1959), 164--177.

\bibitem{Walter1998}
W. Walter. Ordinary differential equations.  Graduate Texts in Mathematics, 182. Readings in Mathematics. Springer-Verlag, New York, 1998.

\bibitem{Weng2015}
S. Weng. A new formulation for the 3-D Euler equations with an application to subsonic flows in a cylinder. Indiana Univ. Math. J. 64  (2015), no. 6, 1609--1642.

\bibitem{XinYin2005}
Z. Xin; H. Yin. Transonic shock in a nozzle. I. Two-dimensional case. Comm. Pure Appl. Math.  58  (2005),  no. 8, 999--1050.

\bibitem{Yuan2006}
H. Yuan. On transonic shocks in two-dimensional variable-area ducts for steady Euler system. SIAM J. Math. Anal. 38 (2006), no. 4, 1343--1370.

\bibitem{Yuan2007}
H. Yuan. Transonic shocks for steady Euler flows with cylindrical symmetry. Nonlinear Anal.  66  (2007),  no. 8, 1853--1878.


\bibitem{Yuan2008-2}
H. Yuan. A remark on determination of transonic shocks in divergent nozzles for steady compressible Euler flows. Nonlinear Anal. Real World Appl.  9  (2008),  no. 2, 316--325.


\bibitem{YuanZhao}
H. Yuan; Q. Zhao. Subsonic flow passing a duct for three-dimensional steady compressible Euler system with friction. arXiv:1711.11431v1. Nov. 28, 2017.
\end{thebibliography}
\end{document}